\documentclass[10.5pt,a4paper]{article}

\usepackage{graphicx,latexsym,euscript,makeidx,color,bm}
\usepackage{amsmath,amsfonts,amssymb,amsthm,thmtools,mathrsfs,enumerate}
\usepackage[colorlinks,linkcolor=blue,anchorcolor=green,citecolor=red]{hyperref}

\usepackage{geometry}
   \def\cA{{\cal A}}

\def\dbE{\mathbb{E}}     
\def\dbF{\mathbb{F}}   \def\cF{{\cal F}}  
     
\def\dbH{\mathbb{H}}

   \def\cL{{\cal L}}  
   \def\cM{{\cal M}}  
   \def\cN{{\cal N}}  
   \def\cO{{\cal O}}  
\def\dbP{\mathbb{P}}     
   \def\cQ{{\cal Q}}  
\def\dbR{\mathbb{R}}     
\def\dbS{\mathbb{S}}     
   \def\cT{{\cal T}}  
   \def\cU{{\cal U}}

\def\ss{\smallskip}      \def\lt{\left}       \def\hb{\hbox}
\def\ms{\medskip}        \def\rt{\right}      \def\ae{\hbox{\rm a.e.}}
        \def\lan{\langle}    \def\as{\hbox{\rm a.s.}}
   \def\ran{\rangle}    \def\tr{\hbox{\rm tr$\,$}}
\def\ts{\textstyle}      \def\llan{\lt\lan}   
\def\no{\noindent}       \def\rran{\rt\ran}   
     \def\esssup{\mathop{\rm esssup}}
     \def\essinf{\mathop{\rm essinf}}
     
\def\rf{\eqref}            \def\hp{\hphantom}
\def\deq{\triangleq}     \def\({\Big (}       \def\nn{\nonumber}
\def\les{\leqslant}      \def\){\Big )}       
\def\ges{\geqslant}      \def\[{\Big[}        
\def\ti{\tilde}          \def\]{\Big]}        
      \def\q{\quad}        
\def\h{\widehat}         \def\qq{\qquad}      \def\1n{\negthinspace}
\def\cd{\cdot}           \def\2n{\1n\1n}      \def\3n{\1n\1n\1n}

\def\a{\alpha}           \def\g{\gamma}   \def\Om{\Omega}  \def\om{\omega}
\def\b{\beta}            \def\d{\delta}        
             \def\Si{\Sigma}  \def\si{\sigma}
\def\e{\varepsilon}   \def\L{\Lambda}  \def\l{\lambda}        
    \def\t{\tau}     \def\f{\varphi}  \def\i{\infty}   

\def\ba{\begin{array}}                \def\ea{\end{array}}
\def\bel{\begin{equation}\label}      \def\ee{\end{equation}}

\newtheorem{theorem}{Theorem}[section]

\newtheorem{proposition}[theorem]{Proposition}

\newtheorem{remark}[theorem]{Remark}
\newtheorem{example}[theorem]{Example}
\newtheorem{taggedassumptionx}{}
\newenvironment{taggedassumption}[1]
 {\renewcommand\thetaggedassumptionx{#1}\taggedassumptionx}
 {\endtaggedassumptionx}
\newtheorem{taggedthmx}{}
\newenvironment{taggedthm}[1]
 {\renewcommand\thetaggedthmx{#1}\taggedthmx}
 {\endtaggedthmx}
\makeatletter
   
   \@addtoreset{equation}{section}
\makeatother
  \allowdisplaybreaks[4]

\begin{document}

\title{\bf Linear-Quadratic Optimal Control for Backward Stochastic Differential Equations with Random Coefficients}
\author{Jingrui Sun\thanks{Department of Mathematics, Southern University of Science and Technology,
                           Shenzhen, 518055, China (Email: {\tt sunjr@sustech.edu.cn}).
                           This author is supported by NSFC grant 11901280, Guangdong Basic and Applied Basic Research Foundation 2021A1515010031, and SUSTech start-up funds Y01286128 and Y01286228.}  \q and\q
         Hanxiao Wang\thanks{Corresponding author. Department of Mathematics, National University of Singapore,
                           Singapore 119076, Singapore (Email: {\tt hxwang14@fudan.edu.cn}).
                           This author is supported by Singapore MOE AcRF Grants R-146-000-271-112.}
                           }

\maketitle

{\no\bf Abstract.}
This paper is concerned with a linear-quadratic (LQ, for short) optimal control problem
for backward stochastic differential equations (BSDEs, for short),
where the coefficients of the backward control system and the weighting matrices in the
cost functional are allowed to be random.
By a variational method, the optimality system, which is a coupled linear forward-backward
stochastic differential equation (FBSDE, for short), is derived, and by a Hilbert space method,
the unique solvability of the optimality system is obtained.
In order to construct the optimal control, a new stochastic Riccati-type equation is introduced.
It is proved that an adapted solution (possibly non-unique) to the Riccati equation exists
and decouples the optimality system.
With this solution, the optimal control is obtained in an explicit way.

\ms

{\no\bf Keywords.}
linear-quadratic optimal control,  backward stochastic differential equation,
random coefficient, stochastic Riccati equation.

\ms

{\no\bf AMS subject classifications.} 93E20, 49N10, 60H10.

\section{Introduction}\label{Sec:Introduction}

Let $(\Om,\cF,\dbP)$ be a complete probability space on which a standard one-dimensional Brownian motion
$W=\{W(t);t\ges0\}$ is defined,
and let $\dbF=\{\cF_t\}_{t\ges0}$ be the usual augmentation of the natural filtration generated by $W$.
For a random variable $\xi$, we write $\xi\in\cF_t$ if $\xi$ is $\cF_t$-measurable,
and for a process $\f$, we write $\f\in\dbF$ if it is $\dbF$-progressively measurable.
For a matrix $H=(h_{ij})\in\dbR^{k\times l}$, we use $|H|$ to denote the Frobenius norm of $H$, that is,
$|H|=(\sum_{i,j}|h_{ij}|^2)^{1\over 2}$.
Let $T>0$ be a fixed time horizon and $\dbH$ be a subset of $\dbR^{k\times l}$. For $t\in[0,T]$, we let
\begin{align*}
L_{\cF_t}^2(\Om;\dbH)
   &= \big\{\xi:\Om\to\dbH~|~\xi\in\cF_t~\hb{and}~\dbE|\xi|^2<\i\big\},\\
L_\dbF^\i(t,T;\dbH)
   &= \big\{\f:[t,T]\times\Om\to\dbH~|~\f\in\dbF~\hb{and is bounded}\big\}.
\end{align*}

\ss

Consider the controlled linear backward stochastic differential equation (BSDE, for short):
\bel{state}\left\{\begin{aligned}
   dY(s) &= \big\{A(s)Y(s)+ B(s)u(s)+C(s)Z(s)\big\}ds  + Z(s)dW(s), \\
    Y(T) &= \xi,
\end{aligned}\right.\ee
where $A,C\in L_\dbF^\i(0,T;\dbR^{n\times n})$ and $B\in L_\dbF^\i(0,T;\dbR^{n\times m})$,
called the {\it coefficients} of the {\it state equation} \rf{state}, are given processes;
$u:[0,T]\times\Om\to\dbR^m$, called a {\it control process}, is selected from a certain space
to influence the {\it state process} $(Y,Z)$;
and $\xi\in L_{\cF_T}^2(\Om;\dbR^n)$, called an {\it terminal state}, is a given random variable.
According to the standard result of BSDEs (see \cite{Yong-Zhou 1999}, for example), the state equation
\rf{state} admits a unique adapted solution $(Y,Z)\equiv (Y^{\xi,u},Z^{\xi,u})$ over $[t,T]$ whenever
the control $u$ is square-integrable over $[t,T]$, i.e., whenever $u$ belongs to the following space:
$$ \cU[t,T]\equiv L_{\dbF}^2(t,T;\dbR^m)
= \bigg\{\f:[t,T]\times\Om\to\dbR^m\bigm|\f\in\dbF~\hb{and}~\dbE\int^T_t|\f(s)|^2ds<\i\bigg\}. $$
Moreover, there exists a constant $K>0$, independent of $(t,x)$ and $u$, such that
$$ \dbE\bigg[\sup_{t\les s\les T}|Y(s)|^2 + \int_t^T|Z(s)|^2ds\bigg]
\les K\dbE\bigg[|\xi|^2 + \int^T_t|u(s)|^2ds\bigg].$$

\ss

Let $\dbS^n$ be the set of symmetric $n\times n$ real matrices, and let $\dbS_+^n$ be the subset of
$\dbS^n$ consisting of positive semi-definite matrices.
To measure the performance of the control process $u$ over $[t,T]$, we introduce the following
quadratic {\it cost functional}
\begin{align}\label{cost}
J(t,\xi;u) &= \dbE\bigg\{\lan G_tY(t),Y(t)\ran +\int_t^T\[\lan Q(s)Y(s),Y(s)\ran+\lan N(s)Z(s),Z(s)\ran \nn\\
           &\hp{=\dbE\bigg\{\lan G_tY(t),Y(t)\ran+\int_t^T\[~} +\lan R(s)u(s),u(s)\ran \]ds\bigg\},
\end{align}
where $G_t:\Om\to\dbS^n$ is a bounded $\cF_t$-measurable random variable, $Q,N\in L^\i_{\dbF}(0,T;\dbS_+^{ n})$,
$R\in L^\i_{\dbF}(0,T;\dbS_+^{m})$, and $\lan\cd\,,\cd\ran$ denotes the Frobenius inner product of two matrices.
With the state equation \rf{state} and the cost functional \rf{cost}, the {\it backward linear-quadratic
(LQ, for short) optimal control problem} can be stated as follows.

\begin{taggedthm}{Problem (BLQ)}
For given terminal state $\xi\in L_{\cF_{T}}^2(\Om;\dbR^n)$, find a control $u^*\in\cU[t,T]$ such that
\bel{inf} J(t,\xi; u^*) = \inf_{u\in\cU[t,T]}J(t,\xi;u) \equiv V(t,\xi). \ee
\end{taggedthm}

A control $u^*\in\cU[t,T]$ satisfying \rf{inf} is called an {\it optimal control} of Problem (BLQ)
for the terminal state $\xi$;
the corresponding state process $(Y^*,Z^*)\equiv(Y^{\xi,u^*},Z^{\xi,u^*})$ is called an {\it optimal state process};
the three-tuple $(Y^*,Z^*,u^*)$ is called an {\it optimal triple};
and the function $V$ is called the {\it value function} of Problem (BLQ).

\ms

The (forward) stochastic LQ problem (SLQ problem, for short) is a classical and fundamental problem in control theory,
which was initially studied by Wonham \cite{Wonham 1968} in 1968.
In the book \cite[Chapter 6]{Yong-Zhou 1999}, it shows that under the {\it standard conditions},
the SLQ problem with deterministic coefficients is (uniquely) solvable;
the  {\it optimality system}, which is a coupled forward-backward stochastic differential equation, can be decoupled;
and the optimal control can be further represented as a linear feedback of the current state,
in terms of the solution to the associated Riccati equation.
Another feature of SLQ problems is that the positive definiteness assumption on the weighting matrix of the control
is not necessary for its solvability.
This type of SLQ problems is called {\it indefinite} SLQ problems, see Chen--Li--Zhou \cite{Chen-Li-Zhou 1998} and some follow-up works
\cite{Chen-Zhou 2000,Rami-Moore-Zhou 2001,Sun-Li-Yong 2016,Wang-Sun-Yong 2019}.
When the coefficients and weighting matrices are allowed to be random,
the associated stochastic Riccati equation becomes a fully nonlinear BSDE with quadratic growth (see \cite{Bismut 1976}),
whose solvability was collected by Peng \cite{Peng 1999} in his list of open problems on BSDEs.
Since then many efforts have been devoted to the SLQ problem with random coefficients;
see Chen--Yong \cite{Chen-Yong 2001}, Kohlmann--Tang \cite{Kohlmann-Tang 2003},
Tang \cite{Tang 2003,Tang 2015}, Sun--Xiong--Yong\cite{Sun-Xiong-Yong 2018}, for instance.
It is particularly worthy to mention that in \cite{Tang 2003},
Tang first obtained a satisfactory solvability result for this type of stochastic Riccati equation under the standard conditions,
which serves as a foundation for our study on Problem (BLQ).
For more detailed history of SLQ problems, we refer the reader to the books by
Yong--Zhou \cite{Yong-Zhou 1999} and Sun--Yong \cite{Sun-Yong 2020}.

\ms

The LQ optimal control problem for BSDEs with deterministic coefficients and weighting matrices was initially studied by Lim--Zhou \cite{Lim-Zhou 2001} in 2001.
The theory of backward LQ optimal control problem has important applications in mathematical finance,
especially in financial investment problems with future conditions (as random variables) specified;
see, \cite{Lim 2004, Lim-Zhou 2001,Peng 1997, Yong-Zhou 1999}.
It also has a great potential in studying stochastic differential games, as a backward LQ optimal
control problem arises naturally when we consider the game in a leader-follower manner;
see for example, \cite{Yong 2002}.
Because of these facts, there has been extensive research on the LQ optimal control (and game) problems for BSDEs.
See, for example, Huang--Wang--Wu \cite{Huang 2016}, Wang--Xiao--Xiong \cite{Wang-Xiao-Xiong 2018},
Du--Huang--Wu \cite{Du 2018}, Li--Sun--Xiong \cite{Li-Sun-Xiong 2019}, and Bi--Sun--Xiong \cite{Bi 2019}.
Note that in our Problem (BLQ), the coefficients of \rf{state} and the weighting matrices in \rf{cost}
are allowed to be random.
This feature makes it more complicated and difficult to study.
Since it was proposed by Lim--Zhou in \cite{Lim-Zhou 2001}, there are few significative results
on Problem (BLQ) with random coefficients so far.

\ms

One difficulty in solving random-coefficient backward LQ Problems is the solvability of the
{\it stochastic Riccati equation}:
\bel{BLQ-Riccati-Equation1}\left\{\begin{aligned}
d\Si(s) &= \[\Si A^\top+A\Si+\Si Q\Si-BR^{-1}B^\top+\L N(I_n+\Si N)^{-1}\L-C(I_n+\Si N)^{-1}\Si C^\top\\
        &\hp{=\[} -C(I_n+\Si N)^{-1}\L-\L(I_n+N\Si)^{-1}C^\top\]ds-\L dW(s),\q s\in[0,T],\\
 \Si(T) &= 0,
\end{aligned}\right.\ee
where the argumet $s$ is suppressed for notational simplicity (and we will frequently do so in the sequel
if no confusion occurs).
Different from the deterninistic case studied in \cite{Lim-Zhou 2001},  equation \rf{BLQ-Riccati-Equation1}
is now a fully nonlinear BSDE with quadratic growth in $\L$.
Due to the presence of $\L$ and $(I_n+N\Si)^{-1}$, the  perturbed method used in \cite{Lim-Zhou 2001} cannot be easily applied to our Problem (BLQ).
To be more convincing, let us elaborate the difficulty we encountered in detail.
It is very hard to directly establish the solvability of \rf{BLQ-Riccati-Equation1}.
However, thanks to the results on solvability of the stochastic Riccati equation for forward LQ
optimal control problems (see \cite{Tang 2003,Sun-Xiong-Yong 2018}), the following perturbed equation
is easily seen to have a unique adapted solution $(\Si_\e,\L_\e)$ for each $\e>0$ (see \rf{def-Si-e}):
\bel{BLQ-Riccati-Equation-e1}\left\{\begin{aligned}
d\Si_\e(s) &= \[\Si_\e A^\top \!\!+\! A\Si_\e \!+\! \Si_\e Q\Si_\e \!-\! BR^{-1}B^\top
              \!\!+\! \L_\e N(I_n\!+\!\Si_\e N)^{-1}\L_\e \!-\! C(I_n\!+\!\Si_\e N)^{-1}\Si_\e C^\top \\
&~\hp{=\[} -C(I_n+\Si_\e N)^{-1}\L_\e-\L_\e(I_n+N\Si_\e)^{-1}C^\top\]ds -\L_\e dW(s), \\
\Si_\e(T)&=\e I_n.
\end{aligned}\right.\ee
The first componet $\Si_\e$ of the solution is bounded, positive semi-definite, and monotone in $\e$.
So the limit $\Si(s)\deq\lim_{\e\to0}\Si_\e(s)$ exits for \ae  ~$s\in[0,T]$, \as~
If all the coefficients are deterministic, then $\L_\e$ is identically zero for every $\e>0$ and hence
$(\Si,0)$ is a solution of \rf{BLQ-Riccati-Equation1}.
But in the random-coefficient case, the $\L_\e$'s are nonzero stochastic processes, even not bounded.
The standard stability estimate also fails here since \rf{BLQ-Riccati-Equation-e1} is a multi-dimensional
quadratic BSDE. The convergence of $\L_\e$ therefore becomes very unclear.
Instead of stubbornly proving the convergence of $\L_\e$, we shall combine the above perturbed approach
with a new method of undetermined coefficients to obtain the existence of a $\L$, which, together with
the limit of $\Si_\e$, gives the adapted solution of \rf{BLQ-Riccati-Equation1}.
We point out that although our work benefits from the results of Lim--Zhou \cite{Lim-Zhou 2001}
and Tang \cite{Tang 2003} a lot,
it is by no means a simple extension of the backward LQ problem from deterministic case to the random one.

\ms

On the other hand, for deterministic-coefficient backward LQ Problems, it has been shown in \cite{Lim-Zhou 2001} and
\cite{Li-Sun-Xiong 2019} that, in order to construct the optimal control, besides an ordinary Riccati equation,
one need also consider an associated uncontrolled BSDE with bounded deterministic coefficients.
Another difficulty is that for our Problem (BLQ), we have to consider the solvability of a BSDE with
{\it unbounded random coefficients}.
To our best knowledge, there are only a few papers dealing with such kind of BSDEs, and no existing results
ensure the existence of an adapted solution to the BSDE associated with Problem (BLQ).

\ms

The purpose of this paper is to overcome the above difficulties and to give a complete solution of Problem (BLQ)
under the following condition: For some $\d>0$,
\bel{condition} Q(s)\ges 0,\q N(s)\ges \d I_n,\q R(s)\ges\d I_m,\q s\in[0,T]. \ee
We shall show that Problem (BLQ) is uniquely solvable and establish the global solvability of the
stochastic Riccati equation \rf{BLQ-Riccati-Equation1} under the condition \rf{condition}.
With the adapted solution to \rf{BLQ-Riccati-Equation1}, we further introduce a decoupled system of
forward-backward stochastic differential equations (FBSDEs, for short) with unbounded random coefficients
and establish its unique solvability.
Then we provide an explicit representation for the unique optimal control of Problem (BLQ), in terms of
the solutions to \rf{BLQ-Riccati-Equation1} and the decoupled system of FBSDEs.

\ms

The rest of this paper is organized as follows.
In Section \ref{Preliminaries}, we collect some preliminary results of BSDEs.
Section \ref{sec3} is devoted to deriving the optimality system for Problem (BLQ)
and establishing its unique solvability.
To decouple the optimality system, we introduce a new stochastic Riccati-type equation
and a decoupled system of FBSDEs with unbounded coefficients in Section \ref{sec4}.
Finally, we establish the solvabilities of the stochastic Riccati equation and the decoupled
system of FBSDEs in Section \ref{sec5}.

\section{Preliminaries}\label{Preliminaries}

Throughout this paper, $\dbR^{n\times m}$ is the Euclidean space consisting of $n\times m$ real matrices,
endowed with the Frobenius inner product $\lan M,N\ran\deq\tr[M^\top N]$, where $M^\top$ and $\tr(M)$
stand for the transpose and the trace of $M$, respectively.
The identity matrix of size $n$ is denoted by $I_n$.
When $m=1$, we simply write $\dbR^{n\times 1}$ as $\dbR^n$.
If there is no confusion, we shall use $\lan\cd\,,\cd\ran$ for inner products in possibly
different Hilbert spaces and denote by $|\cd|$ the norm induced by $\lan\cd\,,\cd\ran$.
Besides the notation introduced in Section \ref{Sec:Introduction}, the following notation will be also
frequently used in this paper:
\begin{align*}
L_{\cF_t}^\i(\Om;\dbH)
  &= \big\{\xi:\Om\to\dbH~|~\xi\in\cF_t~\hb{is bounded}\big\};\\
L_\dbF^\i(\Om;C([t,T];\dbH))
  &= \ts\Big\{\f:[t,T]\times\Om\to\dbH~|~\f\in\dbF~\hb{is continuous and bounded}\Big\};\\
L^1_\dbF(\Om;L^2(t,T;\dbH))
  &= \ts\Big\{\f:[t,T]\times\Om\to\dbH~|~\f\in\dbF~\hb{and}~\dbE\big[\int_t^T|\f(s)|^2ds\big]^{1\over2}<\i\Big\};\\
L_\dbF^2(t,T;\dbH)
  &= \ts\Big\{\f:[t,T]\times\Om\to\dbH~|~\f\in\dbF~\hb{and}~\dbE\int_t^T|\f(s)|^2ds<\i\Big\};\\
L_\dbF^2(\Om;C([t,T];\dbH))
  &= \ts\Big\{\f:[t,T]\times\Om\to\dbH~|~\f\in\dbF~\hb{is continuous and } \dbE\[\sup_{t\les s\les T}|\f(s)|^2\]<\i\Big\}.
\end{align*}
For $M,N\in\dbS^n$, we use the notation $M\ges N$ (respectively, $M>N$) to indicate that $M-N$ is positive
semi-definite (respectively, positive definite).
Further, for an $\dbS^n$-valued measurable function $F$ on $[t,T]$, we write
\begin{alignat*}{3}
& F \ges 0  &&\q \hb{if}~  F(s)\ges 0,       &&\q \ae~s\in[t,T],\\
& F   >  0  &&\q \hb{if}~  F(s)   > 0,       &&\q \ae~s\in[t,T],\\
& F \gg  0  &&\q \hb{if}~  F(s)\ges \d I_n,  &&\q \ae~s\in[t,T],~\hb{for some}~\d>0.
\end{alignat*}
We will say that $F$ is uniformly positive definite if $F\gg0$.

\ms

For the state system \rf{state} and the cost functional \rf{cost}, we impose the following assumptions.

\begin{taggedassumption}{(H1)}\label{ass:H1}\rm
The coefficients of the state equation \rf{state} satisfy
$$ A,C\in L^\i_{\dbF}(0,T;\dbR^{n\times n}), \q B\in L^\i_{\dbF}(0,T;\dbR^{n\times m}).$$
\end{taggedassumption}

\begin{taggedassumption}{(H2)}\label{ass:H2}\rm
The weighting coefficients in the cost functional \rf{cost} satisfy
$$ Q,N\in L^\i_{\dbF}(0,T;\dbS_+^n),
\q R\in L^\i_{\dbF}(0,T;\dbS_+^m),\q G_t\in L_{\cF_t}^\i(\Om;\dbS_+^n).$$
Moreover, there exists a constant $\d>0$ such that
$$ R(s)\ges\d I_m,\q\ae~s\in[0,T],~\as $$
\end{taggedassumption}

We now present a result concerning the well-posedness of the state equation \rf{state}.

\begin{theorem}\label{lmm:well-posedness-BSDE}
Let {\rm\ref{ass:H1}} hold. Then for any terminal state $\xi\in L_{\cF_{T}}^2(\Om;\dbR^n)$
and control $u\in\cU[t,T]$, state equation \rf{state} admits a unique adapted solution
$$
(Y,Z)\equiv (Y^{\xi,u},Z^{\xi,u})\in L_\dbF^2(\Om;C([t,T];\dbR^n))\times L_{\dbF}^2(t,T;\dbR^n).
$$
Moreover, there exists a constant $K>0$, independent of $t,\xi$ and $u$, such that
\begin{align}
\label{well-BSDE-E}
\dbE\bigg[\sup_{t\les s\les T}|Y(s)|^2+\int_t^T|Z(s)|^2ds\bigg]
   &\les K\dbE\bigg[|\xi|^2+ \int^T_t|u(s)|^2ds\bigg], \\
\label{well-BSDE-Et}
\sup_{t\les s\les T}\dbE_t|Y(s)|^2 + \dbE_t\int_t^T|Z(s)|^2ds
   &\les K\dbE_t\bigg[|\xi|^2+ \int^T_t|u(s)|^2ds\bigg],\q\as,
\end{align}
where $\dbE_t[\,\cd\,]=\dbE[\,\cd\,|\,\cF_t]$ is the conditional expectation operator.
\end{theorem}

\begin{proof}
It is standard to obtain the existence and uniqueness of the adapted solution to \rf{state}
and the estimate \rf{well-BSDE-E}.
The  details and proofs of this reslut can be found in \cite{Yong-Zhou 1999}.
We only sketch the proof of estimate \rf{well-BSDE-Et} here.
Applying It\^{o}'s formula to $s\mapsto |Y(s)|^2$ yields that
\begin{align}
\nn&|Y(s)|^2+\int^T_s|Z(r)|^2dr+2\int^T_s Y(r)^\top Z(r)dW(r)\\
\label{lamma-bsde-ito1}
&\q= |\xi|^2+ 2\int^T_s Y(r)^\top\big[A(r)Y(r)+B(r)u(r)+C(r)Z(r)\big]dr,
                           \q s\in[t,T].
\end{align}
Taking conditional expectations with respect to $\cF_t$
on the both sides of \rf{lamma-bsde-ito1} and by \ref{ass:H1}, we get
\begin{align}
&\nn\dbE_t|Y(s)|^2+\dbE_t\int^T_s|Z(r)|^2dr\\
\label{lamma-bsde-ito}&\q\les K \dbE_t|\xi|^2
+K\dbE_t\bigg\{\int^T_s\[|Y(r)|^2+|Y(r)||Z(r)|+|Y(r)||u(r)|\]dr\bigg\},\q s\in[t,T].
\end{align}
By Young inequality, it is clearly seen from \rf{lamma-bsde-ito} that
\begin{align}
\nn&\dbE_t|Y(s)|^2+\dbE_t\int^T_s|Z(r)|^2dr\\
\label{lamma-bsde-ito2}&\q\les K \dbE_t|\xi|^2
+K\dbE_t\bigg\{\int^T_s\[|Y(r)|^2+|u(r)|^2\]dr\bigg\}
+{1\over 2}\dbE_t\int^T_s|Z(r)|^2dr,\q s\in[t,T],
\end{align}
which implies that
\bel{lamma-bsde-ito3}
\dbE_t|Y(s)|^2+\dbE_t\int^T_s|Z(r)|^2dr
\les K \dbE_t|\xi|^2+K\dbE_t\bigg\{\int^T_s\[|Y(r)|^2+|u(r)|^2\]dr\bigg\},
\q s\in[t,T].
\ee
The estimate \rf{well-BSDE-Et} then follows from Gr\"{o}nwall's inequality immediately.
\end{proof}

\ms

Under \ref{ass:H1}, \autoref{lmm:well-posedness-BSDE} shows that
for any $\xi\in L_{\cF_{T}}^2(\Om;\dbR^n)$ and  $u\in\cU[t,T]$,
state equation \rf{state} admits a unique adapted solution
$(Y,Z)\in L_\dbF^2(\Om;C([t,T];\dbR^n))\times L_\dbF^2(t,T;\dbR^n)$.
If, in addition, \ref{ass:H2} holds,
then the random variables on the right-hand side of \rf{cost} are integrable
and hence Problem (BLQ) is well-posed.
When the coefficients and the  weighting matrices reduce to deterministic functions,
 \ref{ass:H1}--\ref{ass:H2} are same as the Assumption (A1) in Lim--Zhou \cite{Lim-Zhou 2001}.
Moreover, \ref{ass:H2} implies the mapping $u\mapsto J(t,\xi;u)$ is uniformly convex,
which plays an important role in establishing the unique solvability of Problem (BLQ).

\section{Optimality Systems and Coupled FBSDEs}\label{sec3}
In this section, we shall  derive the  optimality system  for  the optimal control of Problem (BLQ)
by a variational method
and then study the unique solvability of the optimality system
from a Hilbert space point of view.

\begin{theorem}\label{thm-maximum-principal}
Suppose that {\rm\ref{ass:H1}} and {\rm\ref{ass:H2}} hold.
Then for any given terminal state $\xi\in L^2_{\cF_T}(\Om;\dbR^n)$,
 $u^*\in\cU[t,T]$ is  optimal for {\rm Problem (BLQ)} if and only if
the adapted solution $(Y^*,Z^*,X^*)$ to the following FBSDE
\bel{Y-X-star}
\left\{\begin{aligned}
&dY^*(s)=\big\{A(s)Y^*(s)+ B(s)u^*(s)+C(s)Z^*(s)\big\}ds \\
&\hp{dY^*(s)=}\q + Z^*(s)dW(s),\q s\in[t,T],\\
&dX^*(s)=\big\{-A(s)^\top X^*(s)+ Q(s)Y^*(s)\big\}ds \\
&\hp{dX^*(s)=}\q+ \big\{-C(s)^\top X^*(s)+ N(s)Z^*(s)\big\}dW(s),\q s\in[t,T],\\
&Y^*(T)= \xi,\q X^*(t)=G_tY^*(t)
\end{aligned}\right.
\ee
satisfies the following stationary condition:
\bel{stationary-condition}
R(s)u^*(s)-B(s)^\top X^*(s)=0,\q \ae\, s\in[t,T],\,\,\as
\ee
\end{theorem}

\begin{proof}
By the definition of Problem (BLQ), $u^*$ is an optimal control if and only if
%
%
%
%
\bel{proof-thm-MP-e-op}
J(t,\xi;u^*)\les J(t,\xi;u^*+\e u),\q \forall \,\e\in\dbR,\, u\in\cU[t,T].
\ee
For any fixed but  arbitrary $\e\in\dbR$ and $u\in\cU[t,T]$,
let $(Y_{\e},Z_{\e})$ be the adapted solution of BSDE \rf{state}
corresponding to the terminal state $\xi$ and control $u^*+\e u$; that is
\bel{state-e}\left\{\begin{aligned}
dY_\e(s) &= \big\{A(s)Y_\e(s)+ B(s)(u^*(s)+\e u(s))+C(s)Z_\e(s)\big\}ds \\
         &\qq + Z_\e(s)dW(s), \q s\in[t,T],\\
Y_\e(T) &= \xi.
\end{aligned}\right.\ee
Let $(Y,Z)$ be the adapted solution of the following BSDE:
\bel{Y-u}
\left\{\begin{aligned}
dY(s)&=\big\{A(s)Y(s)+ B(s)u(s)+C(s)Z(s)\big\}ds \\
     &\qq+ Z(s)dW(s),\q s\in[t,T],\\         %
Y(T)&=0.
\end{aligned}\right.
\ee
By the linearity of BSDEs \rf{state-e}, \rf{Y-u}, \rf{Y-X-star}
and the uniqueness of the adapted solution to BSDE \rf{state-e}, we get
\bel{Y-u-Y-star-Y}
Y_\e=Y^*+\e Y,\q Z_\e=Z^*+\e Z.
\ee
Then it is straightforward to deduce the following representation
of the difference $J(t,\xi;u^*+\e u)-J(t,\xi;u^*)$:
\begin{align}
\nn&J(t,\xi;u^*+\e u)-J(t,\xi;u^*)\\
\nn&\q= \e^2\dbE\bigg\{\big\lan G_t Y(t),Y(t)\big\ran +\int_t^T\(\lan QY,Y\ran+\lan NZ,Z\ran+\lan Ru,u\ran \)ds\bigg\}\\
&\label{J-e-J}\hp{\q=} +2\e\dbE\bigg\{\big\lan G_t Y^*(t),Y(t)\big\ran
+\int_t^T\(\lan QY^*,Y\ran+\lan NZ^*,Z\ran+\lan Ru^*,u\ran \)ds\bigg\}.
\end{align}
Thus the condition \rf{proof-thm-MP-e-op} is equivalent to
\begin{align}
\nn& \e^2\dbE\bigg\{\big\lan G_t Y(t),Y(t)\big\ran +\int_t^T\(\lan QY,Y\ran+\lan NZ,Z\ran+\lan Ru,u\ran \)ds\bigg\}
+2\e\dbE\bigg\{\big\lan G_t Y^*(t),Y(t)\big\ran \\
\label{J-e-J-op}&\qq\qq +\int_t^T\(\lan QY^*,Y\ran+\lan NZ^*,Z\ran+\lan Ru^*,u\ran \)ds\bigg\}\ges 0,
\q\forall \e\in\dbR,\, u\in\cU[t,T].
\end{align}
It is clearly seen from \ref{ass:H2} that
$$
\dbE\bigg\{\big\lan G_tY(t),Y(t)\big\ran +\int_t^T\(\lan QY,Y\ran+\lan NZ,Z\ran+\lan Ru,u\ran \)ds\bigg\}
\ges 0,\q\forall u\in\cU[t,T].
$$
Note that for any fixed $u$ and $\xi$,
the left-hand term of \rf{J-e-J-op} could be regarded as a quadratic polynomial of the variable $\e$.
Hence, \rf{J-e-J-op} holds if and only if
\bel{J-e-term}
\dbE\bigg\{\big\lan G_tY^*(t),Y(t)\big\ran
+\int_t^T\(\lan QY^*,Y\ran+\lan NZ^*,Z\ran+\lan Ru^*,u\ran \)ds\bigg\}=0,\q\forall u\in\cU[t,T].
\ee
By applying It\^{o}'s formula to $s\mapsto\lan X^*(s),Y(s)\ran$ on $[t,T]$ and then taking expectation, we get
$$
\dbE\big\lan G_tY^*(t),Y(t)\big\ran=\dbE\big\lan X^*(t),Y(t)\big\ran
=-\dbE\int_t^T\(\lan QY^*,Y\ran+\lan B^\top X^*,u\ran+\lan NZ^*,Z\ran \)ds.
$$
Substituting the above into \rf{J-e-term} yields that
\bel{J-e-term-rw}
\dbE\int_t^T\big\lan Ru^*- B^\top X^*,u\big\ran ds=0,\q\forall u\in\cU[t,T],
\ee
which implies that the stationary condition \rf{stationary-condition} holds.
By reversing the above arguments, the sufficiency of \rf{stationary-condition} follows easily.
\end{proof}

The system \rf{Y-X-star}, together with the  stationary condition \rf{stationary-condition},
is referred to as  the {\it optimality system} for Problem (BLQ).
For any given $u^*\in\cU[t,T]$, the system \rf{Y-X-star} is a decoupled FBSDE.
However, note that the optimal control $u^*$ necessarily satisfies the
stationary condition \rf{stationary-condition}, which is equivalent to
\bel{u-X}
u^*(s)=R(s)^{-1}B(s)^\top X^*(s),\q \ae\, s\in[t,T],\,\as
\ee
Substituting the above into \rf{Y-X-star}, the optimality system becomes a coupled FBSDE as follows:
\bel{Y-X-star-nou}\left\{\begin{aligned}
&dY^*(s)=\big\{A(s)Y^*(s)+ B(s)R(s)^{-1}B(s)^\top X^*(s)+C(s)Z^*(s)\big\}ds \\
&\qq\qq\q + Z^*(s)dW(s),\q s\in[t,T],\\
&dX^*(s)=\big\{-A(s)^\top X^*(s)+ Q(s)Y^*(s)\big\}ds \\
&\qq\qq\q + \big\{-C(s)^\top X^*(s)+ N(s)Z^*(s)\big\}dW(s),\q s\in[t,T],\\
&Y^*(T)= \xi,\q X^*(t)=G_tY^*(t).
\end{aligned}\right.\ee
In the subsequent analysis, we shall consider the well-posedness of FBSDE \rf{Y-X-star-nou}.
To begin with, we present a unique solvability result of Problem (BLQ).

\begin{theorem}\label{thm-BLQ-solvability}
Suppose that {\rm\ref{ass:H1}} and {\rm\ref{ass:H2}} hold.
Then for any terminal state $\xi\in L^2_{\cF_T}(\Om;\dbR^n)$,
{\rm Problem (BLQ)} admits a unique optimal control.
\end{theorem}

\begin{proof}
For any $u\in \cU[t,T]$, consider the following BSDE:
\bel{Y-u-0}\left\{\begin{aligned}
   dY^{0,u}(s) &=\big\{A(s)Y^{0,u}(s)+ B(s)u(s)+C(s)Z^{0,u}(s)\big\}ds \\
               &\q + Z^{0,u}(s)dW(s),\qq s\in[t,T],\\
     Y^{0,u}(T)&= 0.
\end{aligned}\right.\ee
By \autoref{lmm:well-posedness-BSDE}, the above BSDE admits a unique adapted solution
$(Y^{0,u},Z^{0,u})\in L_\dbF^2(\Om;C([t,T];\dbR^n))\times L_\dbF^2(t,T;\dbR^n)$.
By the linearity of BSDE \rf{Y-u-0}, we can define two bounded linear operators
$\cL:\cU[t,T]\to L_\dbF^2(\Om;C([t,T];\dbR^n))\times L_\dbF^2(t,T;\dbR^n)$
and $\cM:\cU[t,T]\to L_{\cF_t}^2(\Om;\dbR^n)$ as follows:
\bel{def-cL-cM}
\cL u=(Y^{0,u},Z^{0,u}),\q \cM u=Y^{0,u}(t),\q u\in\cU[t,T].
\ee
Also we can define the linear operators
$\cN:L^2_{\cF_T}(\Om;\dbR^n)\to L_\dbF^2(\Om;C([t,T];\dbR^n))\times L_\dbF^2(t,T;\dbR^n)$
and $\cO:L^2_{\cF_T}(\Om;\dbR^n)\to L_{\cF_t}^2(\Om;\dbR^n)$ as follows:
\bel{def-cN-cO}
\cN\xi=(Y^{\xi,0},Z^{\xi,0}),\q \cO\xi=Y^{\xi,0}(t),\q \xi\in L^2_{\cF_T}(\Om;\dbR^n),
\ee
with $(Y^{\xi,0},Z^{\xi,0})$ being the adapted solution of the following BSDE:
\bel{Y-0-xi}\left\{\begin{aligned}
   dY^{\xi,0}(s) &=\big\{A(s)Y^{\xi,0}(s)+C(s)Z^{\xi,0}(s)\big\}ds  + Z^{\xi,0}(s)dW(s),\q s\in[t,T],\\
     Y^{\xi,0}(T)&= \xi.
\end{aligned}\right.\ee
Observe that for any $(\xi,u)\in L^2_{\cF_T}(\Om;\dbR^n)\times\cU[t,T]$,
the sum $(Y^{0,u}+Y^{\xi,0},\,Z^{0,u}+Z^{\xi,0})$ satisfies BSDE \rf{state}.
By the uniqueness of the adapted solution to BSDE \rf{state}, we get
\bel{Y-Y-u0-oxi}
(Y,Z)=(Y^{0,u}+Y^{\xi,0},Z^{0,u}+Z^{\xi,0})=\cL u+\cN\xi.
\ee
In particular, the initial value  $Y(t)$  is given by
\bel{Yt-Yt-u0-oxi}
Y(t)=Y^{0,u}(t)+Y^{\xi,0}(t)=\cM u+\cO\xi.
\ee
Now let $\cA^*$ denote the adjoint operator of a linear operator $\cA$,
and define the bounded linear operator $\cQ:L_\dbF^2(t,T;\dbR^n)\times L_\dbF^2(t,T;\dbR^n)\to
L_\dbF^2(t,T;\dbR^n)\times L_\dbF^2(t,T;\dbR^n)$ by
\bel{def-cQ}
\cQ\deq \begin{pmatrix}Q & 0 \\ 0 & N\end{pmatrix}.
\ee
Then by the representations \rf{Y-Y-u0-oxi}, \rf{Yt-Yt-u0-oxi}, and \rf{def-cQ},
the cost functional \rf{cost} can be rewritten as follows:
\begin{align}
\nn J(t,\xi;u)&=\dbE\Bigg\{\lan G_t Y(t),Y(t)\ran
                 +\int_t^T\Bigg[\llan\begin{pmatrix}Q (s)& 0\\ 0 & N(s)\end{pmatrix}
                                 \begin{pmatrix}Y(s) \\ Z(s)\end{pmatrix},
                                 \begin{pmatrix}Y(s) \\ Z(s)\end{pmatrix}\rran \\
\nn&\hp{=\dbE\Bigg\{\lan G_t Y(t),Y(t)\ran+\int_t^T\Bigg[}
                 +\lan R(s)u(s),u(s)\ran \Big]ds\Big\}\\
\nn&=\big\lan G_t (\cM u+\cO\xi),\cM u+\cO\xi\big\ran+\big\lan \cQ(\cL u+\cN\xi),\cL u+\cN\xi\big\ran
      +\big\lan Ru,u\big\ran\\
\nn&=\big\lan (\cM^*G_t\cM+\cL^*\cQ \cL+R) u,u\big\ran+2\big\lan (\cO^*G_t\cM+\cN^*\cQ \cL) u,\xi\big\ran\\
\label{cost-functional-rewrite}&\hp{=\,}+\big\lan (\cO^*G_t\cO+\cN^*\cQ \cN) \xi,\xi\big\ran.
\end{align}
Since all the linear operators involved in the above are bounded, the map $u\mapsto J(t,\xi;u)$ is continuous.
Due to the facts that $\cQ\ges 0$, $G_t\ges 0$ and $R\ges \d I_n$ obtained from \ref{ass:H2}, we have
$$
\big\lan (\cM^*G_t\cM+\cL^*\cQ \cL+R) u,u\big\ran\ges\lan R u,u\ran
=\dbE\int_t^T\lan R(s)u(s),u(s)\ran ds\ges\d\dbE\int_t^T|u(s)|^2ds,
$$
which implies the map $u\mapsto J(t,\xi;u)$ is strictly convex,  and that
$$
J(t,\xi;u)\to\i\q \hbox{as}\q \dbE\int_t^T|u(s)|^2ds\to\i.
$$
Therefore, by the basic theorem in convex analysis,
the unique solvability of
Problem (BLQ), for any given terminal state $\xi\in L^2_{\cF_T}(\Om;\dbR^n)$, is obtained.
\end{proof}

Combining \autoref{thm-maximum-principal} with \autoref{thm-BLQ-solvability} together,
we get the unique solvability  of FBSDE \rf{Y-X-star-nou} immediately.

\begin{theorem}\label{thm-solvability-FBSDE}
Suppose that {\rm\ref{ass:H1}} and {\rm\ref{ass:H2}} hold.
Then for any terminal state $\xi\in L^2_{\cF_T}(\Om;\dbR^n)$,
the coupled FBSDE \rf{Y-X-star-nou} admits a unique adapted solution
$(Y^*,Z^*,X^*)\in L_\dbF^2(\Om;C([t,T];\dbR^n))\times L_{\dbF}^2(t,T;\dbR^n)\times L_\dbF^2(\Om;C([t,T];\dbR^n))$.
Moreover,  the unique optimal control of {\rm Problem (BLQ)} for $\xi$ is given by
\bel{u-X1}
u^*(s)=R(s)^{-1}B(s)^\top X^*(s),\q  s\in[t,T].
\ee
\end{theorem}

\begin{remark}\rm
We emphasize that in FBSDE \rf{Y-X-star-nou},
the terminal state $Y^*(T)$ is an arbitrary $\cF_T$-measurable random vector
and the initial state $X^*(t)$ is determined by the initial value of $Y^*$,
due to which   FBSDE \rf{Y-X-star-nou} is not Markovian even if its coefficients are deterministic functions.
Thus the form of FBSDE \rf{Y-X-star-nou} is a little different from the standard  FBSDEs
(see \cite{Ma-Yong 1999}, for example).
In particular, it will be interesting to give a ``Four-Step Scheme" for this type of FBSDEs.
\end{remark}

\section{Decoupling, Riccati equation,  BSDE and FSDE with unbounded coefficients}\label{sec4}
Since the optimality system is a fully coupled FBSDE,
it usually becomes difficult to find the optimal control by solving \rf{Y-X-star-nou} directly.
Then, to construct an optimal control from the optimality system  \rf{Y-X-star-nou},
a decoupling technique needs to be adopted.
Thus, we now introduce the following stochastic Riccati-type equation:
\bel{BLQ-Riccati-Equation}
\left\{\begin{aligned}
d\Si(s)&=\[\Si A^\top+A\Si+\Si Q\Si-BR^{-1}B^\top+\L N(I_n+\Si N)^{-1}\L-C(I_n+\Si N)^{-1}\Si C^\top\\
       &\hp{=\[}-C(I_n+\Si N)^{-1}\L-\L(I_n+N\Si)^{-1}C^\top\]ds-\L dW(s),\q s\in[0,T],\\
 \Si(T)&=0.
\end{aligned}\right.
\ee
If the above equation is solvable with $(\Si,\L)$ being a solution, we introduce the following BSDE:
\bel{BLQ-BSDE-unbounded}
\left\{\begin{aligned}
d\f(s)&=\big\{(A+\Si Q)\f-C(I_n+\Si N)^{-1}\b+\L N(I_n+\Si N)^{-1}\b\big\}ds\\
      &\hp{=}~-\b dW(s),\qq s\in[t,T], \\
 \f(T)&=-\xi.
\end{aligned}\right.
\ee
It is noteworthy that $\L$ is  merely square-integrable in general,
thus \rf{BLQ-BSDE-unbounded} is a BSDE with unbounded coefficients.
Suppose that BSDE \rf{BLQ-BSDE-unbounded} has a solution $(\f,\b)$,
we consider the following forward stochastic different equation (FSDE, for short):
\bel{BLQ-fSDE}
\left\{
\begin{aligned}
d X(s)&=-\big\{(A^\top+Q\Si)X+Q\f\big\}ds+\big\{-C^\top X+N(I_n+\Si N)^{-1}(\L+\Si C^\top)X\\
&\hp{=-\big\{(A^\top+Q\Si)X+Q\f\big\}ds+\big\{}+N(I_n+\Si N)^{-1}\b\big\}dW(s),\qq s\in[t,T], \\
X(t)&=-(I_n+G_t\Si(t))^{-1}G_t\f(t).
\end{aligned}\right.
\ee
Similar to \rf{BLQ-BSDE-unbounded},
the coefficients of FSDE \rf{BLQ-fSDE} are also unbounded in general.
Under the assumption that the above equations
\rf{BLQ-Riccati-Equation}--\rf{BLQ-BSDE-unbounded}--\rf{BLQ-fSDE} are solvable,
the following result provides a method of decoupling FBSDE \rf{Y-X-star-nou}.
At first, we impose an additional assumption for  the weighting matrix $N$
of cost functional \rf{cost}.

\begin{taggedassumption}{(H3)}\label{ass:H3}
There exist two constants $\d,\l>0$ such that
\bel{ass:H3-1}
\d I_n \les N(s)\les\l I_n,\qq a.s.,\,\, a.e.\, s\in[0,T],
\ee
or, equivalently,
\bel{ass:H3-2}
{1\over \l} I_n \les N(s)^{-1}\les{1\over \d} I_n,\qq a.s.,\,\, a.e.\, s\in[0,T],
\ee
where $N(s)^{-1}$ stands for the inverse of $N(s)$.
\end{taggedassumption}

Since $N$ is assumed to be  bounded in \ref{ass:H2},
the existence of $\l$ can follow from that easily.
The non-degenerate assumption (i.e., $\d I_n \les N(s)$) is a technical condition (see \rf{est-Z1}--\rf{ti-R1} for some reasons).
At the moment, we cannot improve it and we shall  come
back in our future publications.

\begin{theorem}\label{decoupling}
Let {\rm \ref{ass:H1}--\ref{ass:H2}--\ref{ass:H3}} hold.
Suppose that Riccati equation \rf{BLQ-Riccati-Equation}
has a solution $(\Si,\L)\in L^\i_{\dbF}(\Om;C([0,T];\dbS_+^n))\times L^2_{\dbF}(0,T;\dbS^{ n})$ such that
the corresponding decoupled system of BSDE \rf{BLQ-BSDE-unbounded}  and FSDE \rf{BLQ-fSDE} has a solution $(\f,\b,X)\in
L^2_{\dbF}(\Om;C([t,T];\dbR^n))\times L^1_{\dbF}(\Om; L^2(t,T;\dbR^{n}))\times L^2_{\dbF}(\Om;C([t,T];\dbR^n))$.
Then  the unique adapted solution $(Y^*,Z^*,X^*)$ of  FBSDE  \rf{Y-X-star-nou} can be given by
\bel{decoupling-main}
(Y^*,Z^*,X^*)=\big(-\Si X-\f,\, (I_n+\Si N)^{-1}(\L X+\Si C^\top X+\b),\, X\,\big),
\ee
and the unique optimal control $u^*$ of {\rm Problems (BLQ)} has the following  explicit representation:
\bel{decoupling-u}
u^*=R^{-1}B^\top X.
\ee
\end{theorem}

\begin{proof}
For convenience, we denote
\bel{proof-decoupling-hat}
(\h Y,\h Z,\h X)\deq\big(-\Si X-\f,\, (I_n+\Si N)^{-1}(\L X+\Si C^\top X+\b),\, X\,\big).
\ee
Note that
$$
\int_t^T\big|(I_n+\Si N)^{-1}(\L X+\Si C^\top X+\b)\big|^2ds<\i,\q\as
$$
Define for each $k > 1$ the stopping time (with the convention $\inf\emptyset=\i$)
$$
\t_k=\inf\left\{s\in[t,T];\,\int_t^s\big|(I_n+\Si N)^{-1}(\L X+\Si C^\top X+\b)\big|^2dr\ges k\right\}.
$$
By \ref{ass:H3}, we have
\begin{align}
\nn&\dbE\int_t^{\t_k\wedge T}\big|(I_n+\Si N)^{-1}(\L X+\Si C^\top X+\b)\big|^2ds\\
\nn&\q=\dbE\int_t^{\t_k\wedge T}\big|N^{-1}N(I_n+\Si N )^{-1}(\L X+\Si C^\top X+\b)\big|^2ds\\
\label{est-Z1}&\q\les K \dbE\int_t^{\t_k\wedge T}\big|N(I_n+\Si N)(\L X+\Si C^\top X+\b)\big|^2ds.
\end{align}
Recall that $X$ satisfies FSDE \rf{BLQ-fSDE}, we have by It\^{o}'s isometry that
\begin{align}
\nn&\dbE\int_t^{\t_k\wedge T}\big|N(I_n+\Si N)(\L X+\Si C^\top X+\b)\big|^2ds\\
\nn&\q\les K \dbE\big|G_t(I_n+\Si(t)G_t)^{-1}\f(t)\big|^2+ K \dbE|X(T\wedge\t_k)|^2\\
\nn&\hp{\les K}+K\dbE\int_t^{\t_k\wedge T}\[\big|(A^\top+Q\Si)X\big|^2+\big|Q\f\big|^2+\big|C^\top X\big|^2\] ds\\
\nn&\q\les K \dbE\big|G_t(I_n+\Si(t)G_t)^{-1}\f(t)\big|^2+K \dbE\Big[\sup_{s\in[t,T]}|X(s)|^2\Big]\\
\nn&\hp{\les K}+K\dbE\int_t^T\[\big|(A^\top+Q\Si)X\big|^2+\big|Q\f\big|^2+\big|C^\top X\big|^2 \] ds\\
\label{est-Z2}&\q\les K \dbE\Big[\sup_{s\in[t,T]}|X(s)|^2+\sup_{s\in[t,T]}|\f(s)|^2\Big].
\end{align}
Combining the above with \rf{est-Z1}, by the definition of $\h Z$, we have
\begin{align}
\dbE\int_t^{\t_k\wedge T}|\h Z|^2ds
\nn&=\dbE\int_t^{\t_k\wedge T}\big|(I_n+\Si N)^{-1}(\L X+\Si C^\top X+\b)\big|^2ds\\
\label{est-Z}&\les K \dbE\Big[\sup_{s\in[t,T]}|X(s)|^2+\sup_{s\in[t,T]}|\f(s)|^2\Big]<\i.
\end{align}
Since $\lim_{k\to\i}\t_k=\i$ almost surely
and the right-hand side of the above inequality is independent of $\t_k$,
we conclude $\h Z=(I_n+\Si N)(\L X+\Si C^\top X+\b)\in L_{\dbF}^2(t,T;\dbR^n)$ by letting $k\to\i$.
Then it is clear to see that
$(\h Y,\h Z,\h X)\in L^2_{\dbF}(\Om;C([t,T];\dbR^n))\times L^2_{\dbF}(t,T;\dbR^{n})
\times L^2_{\dbF}(\Om;C([t,T];\dbR^n))$.

\ms

By the uniqueness of the adapted solution to FBSDE \rf{Y-X-star-nou} from \autoref{thm-solvability-FBSDE},
it suffices to verify that $(\h Y,\h Z,\h X)$ satisfies the FBSDE \rf{Y-X-star-nou}.
By the definition \rf{proof-decoupling-hat} of $(\h Y,\h Z,\h X)$ and the equation \rf{BLQ-fSDE} of $X$,
$\h X$ satisfies
\begin{align}
\nn d \h X(s)&=\big\{-A^\top\h X-Q[\Si \h X+\f]\big\}ds\\
\nn &\hp{=}+\big\{-C^\top\h X+N[(I_n+\Si N)^{-1}(\L\h X+\Si C^\top\h X+\b)]\big\}dW(s)\\
\label{hat-X(s)}&=\big\{-A^\top\h X+Q\hat Y\big\}ds+\big\{-C^\top\h X+N\h Z\big\}dW(s),\q s\in[t,T].
\end{align}
Further, by the initial condition in \rf{BLQ-fSDE}, the initial value of $\h X$ is given by
\begin{align}
\nn\h X(t)&=X(t)=[I_n+G_t\Si(t)] X(t)-G_t\Si(t)X(t)\\
\nn&=-[I_n+G_t\Si(t)][I_n+G_t\Si(t)]^{-1}G_t\f(t)-G_t\Si(t)X(t)\\
\label{hat-X(t)}&=G_t[-\f(t)-\Si(t)X(t)]=G_t\h Y(t).
\end{align}
Thus  $\h X$ satisfies the same FSDE as $X^*$
with $(Y^*,Z^*)$ replaced by $(\h Y,\h Z)$.

\ms

We now show that $(\h Y,\h Z)$ satisfies the same BSDE as $(Y^*,Z^*)$ in \rf{Y-X-star-nou}
with $X^*$ replaced by $\h X$.
By applying It\^{o}'s formula to $s\mapsto \h Y(s)\equiv-\Si(s)X(s)-\f(s)$,
some  straightforward calculations yield that
\begin{align}
\nn d \h Y(s)&=d[-\Si(s)X(s)-\f(s)]\\
%
%
%
%
%
%
%
\nn&=\Big\{-A(\Si X+\f)+BR^{-1}B^{\top} X+C(I_n+\Si N)^{-1}(\L X+\Si C^\top X+\b)\\
\nn&\hp{=\Big\{}
+\L N (I_n+\Si N)^{-1}\Si C^\top X-\L C^\top X+\L (I_n+N\Si)^{-1}C^\top X\Big\}ds\\
\nn&\hp{=}
+\big\{\L X+\Si C^\top X+\b-\Si N(I_n+\Si N)^{-1}(\L X+\Si C^\top X+\b)\big\}dW(s)\\
\label{hat-Y(s)}&\equiv(I)ds+(II)dW(s).
\end{align}
%
%
%
%
%
%
Using the fact that
\bel{Si-Multi}
(I_n+\Si N)^{-1}\Si=\Si(I_n+N\Si )^{-1},\q N\Si(I_n+N\Si )^{-1}-I_n=-(I_n+N\Si )^{-1},
\ee
we have
$$
\begin{aligned}
&\L N (I_n+\Si N)^{-1}\Si C^\top X-\L C^\top X+\L (I_n+N\Si)^{-1}C^\top X\\
&\q=\L N \Si(I_n+N\Si )^{-1} C^\top X-\L C^\top X+\L (I_n+N\Si)^{-1}C^\top X\\
&\q=\L [N \Si(I_n+N\Si )^{-1} -I_n]C^\top X+\L (I_n+N\Si)^{-1}C^\top X\\
&\q=-\L (I_n+N\Si)^{-1}C^\top X+\L (I_n+N\Si)^{-1}C^\top X=0.
\end{aligned}
$$
Then by the definition \rf{proof-decoupling-hat} of $(\h X,\h Y,\h Z)$,
the drift term in \rf{hat-Y(s)} can be rewritten as
\begin{align}
\nn(I)&=-A(\Si X+\f)+BR^{-1}B^{\top} X+C(I_n+\Si N)^{-1}(\L X+\Si C^\top X+\b)\\
\label{(I)-final}&=A\h Y+ BR^{-1}B^\top\h X+C\h Z.
\end{align}
By the definition \rf{proof-decoupling-hat} of $\h Z$,
the diffusion term in \rf{hat-Y(s)} can be expressed as
\begin{align*}
(II)&=[I_n-\Si N(I_n+\Si N)^{-1}](\L X+\Si C^\top X+\b)\\
    &=(I_n+\Si N)^{-1}(\L X+\Si C^\top X+\b)=\h Z.
\end{align*}
Combining the above with \rf{(I)-final}, we can rewrite \rf{hat-Y(s)} as
$$
d\h Y(s)=\big\{A\h Y+ BR^{-1}B^\top\h X+C\h Z\big\}ds+\h ZdW(s).
$$
Moreover, by the terminal values of $\Si$ and $\f$,
$\h Y$ satisfies the following terminal condition:
$$
\h Y(T)=-\Si(T) X(T)-\f(T)=\xi.
$$
It follows that $(\h Y,\h Z)$  satisfies the same BSDE as $(Y^*,Z^*)$  with $X^*$ replaced by $\h X$.
The proof is thus completed.
\end{proof}

To illustrates the procedure for finding the optimal control by \autoref{decoupling},
we conclude this section by presenting the following simple example.

\begin{example}\rm
Consider the one-dimensional controlled BSDE:
\bel{state-example}\left\{\begin{aligned}
   dY(s) &={1\over 1+W(s)^2} u(s)ds  + Z(s)dW(s),\q s\in[0,T], \\
    Y(T) &= 1,
\end{aligned}\right.\ee
and the cost functional:
\begin{align}\label{cost-example}
J(0,\xi;u) &= \dbE\int_0^T\bigg[{2+W(s)^2\over 1+W(s)^2} |Z(s)|^2  +{2+W(s)^2\over 1+W(s)^2} |u(s)|^2 \bigg]ds.
\end{align}
Notice that
$$0<{1\over 1+W(s)^2}\les 1, \q 1<{2+W(s)^2\over 1+W(s)^2}\les 2,\q s\in[0,T].$$
Thus, the example satisfies  the assumptions \ref{ass:H1}--\ref{ass:H2}--\ref{ass:H3}.
In the following, we are applying \autoref{decoupling} to obtain the optimal control.
The corresponding Riccati equation \rf{BLQ-Riccati-Equation}, BSDE \rf{BLQ-BSDE-unbounded}, SDE \rf{BLQ-fSDE} read:
\bel{BLQ-Riccati-Equation-example}
\left\{\begin{aligned}
d\Si(s)&=\[-{1\over [1+W(s)^2] [2+W(s)^2]}+\L {2+W(s)^2\over 1+W(s)^2}\(I_n+\Si{2+W(s)^2\over 1+W(s)^2}\)^{-1}\L\]ds\\
       &\hp{=\,}-\L dW(s),\q s\in[0,T],\\
 \Si(T)&=0,
\end{aligned}\right.
\ee
\bel{BLQ-BSDE-example}
\left\{\begin{aligned}
d\f(s)&=\L {2+W(s)^2\over 1+W(s)^2}\(I_n+\Si{2+W(s)^2\over 1+W(s)^2}\)^{-1}\b ds-\b dW(s),\qq s\in[t,T], \\
 \f(T)&=-1,
\end{aligned}\right.
\ee
and
\bel{BLQ-fSDE-example}
\left\{
\begin{aligned}
d X(s)&=\bigg\{{2+W(s)^2\over 1+W(s)^2}\(I_n+\Si{2+W(s)^2\over 1+W(s)^2}\)^{-1}\L X\\
&\hp{=-\big\{}+{2+W(s)^2\over 1+W(s)^2}\(I_n+\Si{2+W(s)^2\over 1+W(s)^2}\)^{-1}\b\bigg\}dW(s),\qq s\in[t,T], \\
X(t)&=0.
\end{aligned}\right.
\ee
Note that $(-1,0,0)$ is the unique solution to the decoupled system of BSDE \rf{BLQ-BSDE-example} and SDE \rf{BLQ-fSDE-example}.
By the formula \rf{decoupling-u} in \autoref{decoupling}, the unique optimal control $u^*$ is given by
\bel{decoupling-u-example}
u^*(s)=R(s)^{-1}B(s)^\top X(s)\equiv 0,\q s\in[0,T].
\ee

\end{example}

\section{Solvability of  Riccati equation, BSDE and FSDE with unbounded coefficients}\label{sec5}

If the solutions $(\Si,\L,\f,\b,X)$ of \rf{BLQ-Riccati-Equation}--\rf{BLQ-BSDE-unbounded}--\rf{BLQ-fSDE} are solved,
then it immediately follows from \autoref{decoupling} that FBSDE \rf{Y-X-star-nou} can be decoupled
and the unique optimal control of Problem (BLQ) can be represented explicitly.
In this subsection, we shall establish the solvability of Riccati equation \rf{BLQ-Riccati-Equation},
BSDE \rf{BLQ-BSDE-unbounded}, and FSDE \rf{BLQ-fSDE}.

\ms
We begin with two interesting results of the optimal control problems for forward SDEs,
which will play a basic role in our subsequent analysis.
For any given $\e>0$, consider the following stochastic Riccati equation:
\bel{Riccati-Pe}
\left\{
\begin{aligned}
dP_\e(s)&=-\Big\{P_\e A+A^T P_\e+Q-\big[P_\e(B,C)+(0,\Pi_\e)\big]\\
        &\hp{=-\big\{}\times\begin{pmatrix}R & 0 \\ 0 & N+P_\e\end{pmatrix}^{-1}
        \big[(B,C)^\top P_\e+(0,\Pi_\e)^\top\big]\Bigg\}+\Pi_\e dW(s),\q s\in[0,T],\\
P_\e(T)&=\e^{-1} I_n.
\end{aligned}
\right.
\ee
Note that \rf{Riccati-Pe} is the Riccati equation associated the following forward LQ problems:
For the given $\e>0$, consider the following controlled FSDE:
\bel{state-SDE-e}\left\{\begin{aligned}
   dX(s) &=\big\{A(s)X(s)+ B(s)u(s)+ C(s)v(s)\big\}ds +v(s)dW(s),\q s\in[t,T],\\
     X(t)&= \eta,
\end{aligned}\right.\ee
and the cost functional:
\begin{align}
\nn J_\e(t,\eta;u,v)&= \dbE_t\bigg\{\int_t^T\[\lan Q(s)X(s),X(s)\ran
+\lan N(s)v(s),v(s)\ran+\lan R(s)u(s),u(s)\ran \]ds\\
\label{cost-X-e}&\hp{= \dbE_t\bigg\{}
+\e^{-1}\lan X(T),X(T)\ran \bigg\}.
\end{align}
 For any given $(t,\eta)\in[0,T]\times L^2_{\cF_t}(\Om;\dbR^n)$,
find a $(u^*,v^*)\in L_{\dbF}^2(t,T;\dbR^m)\times L_{\dbF}^2(t,T;\dbR^n)$ such that
$$
J_\e(t,\eta; u^*,v^*)=\essinf_{u,v}J_\e(t,\eta;u,v)=V_\e(t,\eta),
$$
as $(u,v)$ ranges over the space $L_{\dbF}^2(t,T;\dbR^m)\times L_{\dbF}^2(t,T;\dbR^n)$.
Under \ref{ass:H1}--\ref{ass:H2}--\ref{ass:H3}, we have
$$
\begin{pmatrix}R & 0 \\ 0 & N\end{pmatrix}\ges\d I_{m+n},\q Q\ges 0,\q \hbox{and}\q \e^{-1}\ges 0.
$$
Thus the above SLQ problems satisfies the so-called {\it standard condition} in the literature \cite{Yong-Zhou 1999,Tang 2003,Sun-Xiong-Yong 2018}.
Then by \cite[Theorem 6.2.]{Sun-Xiong-Yong 2018} (or \cite[Theorem 5.3]{Tang 2003}),
Riccati equation \rf{Riccati-Pe} admits a unique solution
$(P_\e,\Pi_\e)\in L^\i_{\dbF}(\Om;C([0,T];\dbS^n_+))\times L^2_{\dbF}(0,T;\dbS^n)$.
The following result shows that $P_\e$ is uniformly positive definite (for the given $\e>0$).

\begin{proposition}\label{P-e-UP}
Let {\rm\ref{ass:H1}--\ref{ass:H2}--\ref{ass:H3}} hold.
Then for any given $\e>0$, $P_\e$ is uniformly positive definite; that is
\bel{P-e-UP-main}
P_\e(s)\ges \a_\e I_n,\q \as,\,\, s\in[0,T],
\ee
for some  $\a_\e>0$.
\end{proposition}

\begin{proof}
For the given $\e>0$,  by \ref{ass:H2}--\ref{ass:H3}, the cost functional $J_\e$ \rf{cost-X-e} satisfies
\begin{align}
\nn J_\e(t,\eta;u,v)\ges (\e^{-1}\wedge\d)\dbE_t\bigg\{ |X(T)|^2 +\int_t^T\[|u(s)|^2+|v(s)|^2\]ds\bigg\},\\
\label{cost-X-e-u-X}
\forall (t,\eta)\in[0,T]\times L^2_{\cF_t}(\Om;\dbR^n),\,(u,v)\in L_{\dbF}^2(t,T;\dbR^m)\times L_{\dbF}^2(t,T;\dbR^n).
\end{align}
We now prove the existence of $\a_\e$ by contradiction.
If not, then for any $\a>0$,
there exist a $t\in[0,T)$ and an $\Om_t\in\cF_t$ with $\dbP(\Om_t)>0$ such that
\bel{P-e-t}
\si_{\min}(P_\e(t,\om))\les \a ,\q \as\,\,\om\in\Om_t,
\ee
where $\si_{\min}(P_\e(t,\om))$ stands for the minimal eigenvalue of the symmetric matrix $P_\e(t,\om)$.
Then we can find an $\cF_t$-measurable random vector $\eta_\a$ with $|\eta_\a|=1$ such that
\bel{P-e-t-eta}
\lan P_\e(t,\om)\eta_\a (\om),\eta_\a(\om)\ran=\si_{\min}(P_\e(t,\om))\les \a,\q \as\,\,\om\in\Om_t.
\ee
For the fixed $(t,\eta_\a)$, by \cite[Corollary 5.7 and Theorem 6.7]{Sun-Xiong-Yong 2018},
there exists a control $(u^{*}_\a,v^{*}_\a)\in L_{\dbF}^2(t,T;\dbR^m)\times L_{\dbF}^2(t,T;\dbR^n)$
such that
\bel{P-e-t-eta-u}
\lan P_\e(t)\eta_\a ,\eta_\a\ran=V_\e(t,\eta_\a)=J_\e(t,\eta_\a;u^{*}_\a,v^{*}_\a),\q \as
\ee
Combining  \rf{P-e-t-eta-u} with \rf{cost-X-e-u-X}, we get
\begin{align}
\nn&\dbE_t\bigg\{|X^{*}_\a(T)|^2 +\int_t^T\[|u^{*}_\a(s)|^2+|v^{*}_\a(s)|^2\]ds\bigg\}\\
\label{x-u-esti1}&\q \les {1\over (\e^{-1}\wedge\d)} J_\e(t,\eta_\a;u^{*}_\a,v^{*}_\a)
={1\over (\e^{-1}\wedge\d)}\lan P_\e(t)\eta_\a ,\eta_\a\ran,
\end{align}
where $X^*_\a$ is the solution of \rf{state-SDE-e} corresponding to $\eta_\a$ and $(u^{*}_\a,v^{*}_\a)$; that is
\bel{state-SDE-e1}\left\{\begin{aligned}
   dX_\a^*(s) &=\big\{A(s)X^*_\a(s)+ B(s)u_\a^*(s)+ C(s)v_\a^*(s)\big\}ds +v_\a^*(s)dW(s),\q s\in[t,T],\\
     X_\a^*(t)&= \eta_\a.
\end{aligned}\right.\ee
The inequality \rf{x-u-esti1}, together with \rf{P-e-t-eta}, implies that
\bel{x-u-esti}
\dbE_t\bigg\{|X^*_\a(T)|^2 +\int_t^T\[|u^{*}_\a(s)|^2+|v^{*}_\a(s)|^2\]ds\bigg\}
\les {1\over (\e^{-1}\wedge\d)}\a,\qq \as ~~\hbox{on}~~ \Om_t.
\ee
Moreover, from \rf{state-SDE-e1} and the fact $X^*_\a(T)=X^*_\a(T)$, we see that $(X^*_\a,v^{*}_\a)$ also satisfies the following BSDE
(with $(Y^*_\a,Z^*_\a)$ being unknown variables):
\bel{state-SDE-e-BSDE}\left\{\begin{aligned}
   dY^*_\a(s)&=\big\{A(s)Y^*_\a(s)+ B(s)u ^{*}_\a(s)+ C(s)Z^*_\a(s)\big\}ds\\
             &\hp{=}+Z^*_\a(s)dW(s),\qq\q s\in[t,T],\\
     Y^*_\a(T)&= X^*_\a(T).
\end{aligned}\right.\ee
Then, by the estimate \rf{well-BSDE-Et} in \autoref{lmm:well-posedness-BSDE},
there exists a constant $K>0$,  independent of $\a$ such that
\bel{Y-BSDE-X(T)-u}
\sup_{s\in[t,T]}\dbE_t|X^*_\a(s)|^2
=\sup_{s\in[t,T]}\dbE_t|Y^*_\a(s)|^2
\les K\dbE_t\bigg[|X^*_\a(T)|^2+\int_t^T|u^{*}_\a(s)|^2ds\bigg].
\ee
Using \rf{Y-BSDE-X(T)-u}--\rf{x-u-esti} and the fact that $|X^{*}_\a(t)|^2=|\eta_\a|^2=1$,
we get
\begin{align}
\nn 1&=|\eta_\a|^2=|X^{*}_\a(t)|^2\les \sup_{s\in[t,T]}\dbE_t|X^*_\a(s)|^2\\
\label{UP-Final1}&\les K\dbE_t\bigg[|X^*_\a(T)|^2+\int_t^T|u^{*}_\a(s)|^2ds\bigg]
\les {K\over (\e^{-1}\wedge\d)}\a,\qq \as \q\hbox{on}\q \Om_t,
\end{align}
which implies that
\bel{UP-Final}
1\les {K\over (\e^{-1}\wedge\d)}\a.
\ee
By taking  a small enough $\a>0$ such that ${K\over (\e^{-1}\wedge\d)}\a<1$,
we get the contradiction immediately.
\end{proof}

For any $\e\ges 0$, we consider the following controlled FSDE:
\bel{state-ti-X}\left\{\begin{aligned}
   dX_\e(s) &=\ti A_\e(s)X_\e(s)ds+ u(s) dW(s),\q s\in[t,T],\\
     X_\e(t)&= \eta,
\end{aligned}\right.\ee
and cost functional:
\begin{align}
\nn \ti J_\e(t,\eta;u)&=\dbE \int_t^T\[\lan\ti Q_\e(s)X_\e(s),X_\e(s)\ran+2\lan \ti S(s)X_\e(s),u(s)\ran
                                                                      +\lan \ti R(s)u(s),u(s)\ran\]ds\\
\label{cost-ti-X}     &\hp{=}+\dbE\lan \ti G_\e X_\e(T),X_\e(T)\ran.
\end{align}
We introduce the following assumption of the coefficients in  \rf{state-ti-X} and weighting matrices in \rf{cost-ti-X}.

\begin{taggedassumption}{(H4)}\label{ass:H4}
For any $\e\ges0$, the coefficient $\ti A_\e$ and weighting matrices $\ti Q_\e,\ti S,\ti R, \ti G_\e$ satisfy:
$$
\begin{aligned}
&\ti A_\e \in L^\i_{\dbF}(0,T;\dbR^{n\times n}),\q \ti Q_\e \in L^\i_{\dbF}(0,T;\dbS_+^{n}),
                                             \q \ti S \in L^\i_{\dbF}(0,T;\dbR^{m\times n}),\\
& \ti R \in L^\i_{\dbF}(0,T;\dbS_+^{m}), \q \ti G_\e \in L^\i_{\cF_T}(\Om;\dbS_+^{n}),
                                            \q \ti Q_\e-\ti S^\top \ti R^{-1}\ti S\ges 0.
\end{aligned}
$$
Moreover,  there exist two constants $\d,K>0$, independent of $\e$ such that
$$
|\ti A_\e(s)|+|\ti Q_\e(s)|+|\ti G_\e|\les K,\q \ti R(s)\ges\d I_m,
\q\as,\,\,\ae\, s\in[0,T].
$$
\end{taggedassumption}

\ms

With the state equation \rf{state-ti-X} and cost functional \rf{cost-ti-X}, we consider the following  LQ problem:

\ms

\textbf{Problem (SLQ$_\e$).}
For any given $\e\ges 0$ and $(t,\eta)\in[0,T]\times L^2_{\cF_t}(\Om;\dbR^n)$,
find a $u_\e^*\in\cU[t,T]$ such that
\bel{def-V-e}
\ti J_\e(t,\eta; u_\e^*)=\inf_{u\in\cU[t,T]}\ti J_\e(t,\eta;u)=\ti V_\e(t,\eta).
\ee

\ms

The following result is concerned with the stability of the value functions $\{\ti V_\e\}_{\e\ges 0}$.

\begin{proposition}\label{stability-value-function-SLQ}
Let {\rm\ref{ass:H4}} hold.
Suppose that
%
\bel{Prop-VS-1}
\lim_{\e\to 0^+}
\[\big|\ti A_\e(s)-\ti A_0(s)\big|+\big|\ti Q_\e(s)-\ti Q_0(s)\big|+\big|\ti G_\e-\ti G_0\big|\]=0,
\q\as,\,\ae\,s\in[0,T].
\ee
Then for any $(t,\eta)\in[0,T]\times L^2_{\cF_t}(\Om;\dbR^n)$,
the following convergence holds:
\bel{Prop-VS-2}
\lim_{\e\to0^+}\ti V_\e(t,\eta)=\ti V_0(t,\eta).
\ee
\end{proposition}

\begin{proof}
Let $(\bar X_\e,\bar Y_\e,\bar Z_\e)$ and  $(\h X_\e,\h Y_\e,\h Z_\e)$ be
the adapted solutions to the decoupled linear FBSDEs
\bel{bar-X}\left\{\begin{aligned}
&d\bar X_\e(s)=\ti A_\e(s)\bar X_\e(s)ds+ u(s) dW(s),\\
&d\bar Y_\e(s)=-\big\{\ti A_\e(s)^\top\bar Y_\e(s)+\ti Q_\e(s)\bar X_\e(s)+\ti S(s)^\top u(s)\big\}ds
                                                                          +\bar Z_\e(s)dW(s),\\
& \bar X_\e(t)=0,\qq \bar Y_\e(T)= \ti G_\e \bar X_\e(T),
\end{aligned}\right.\ee
and
\bel{hat-X}\left\{\begin{aligned}
&d\h X_\e(s)=\ti A_\e(s)\h X_\e(s)ds,\\
&d\h Y_\e(s)= -\big\{\ti A_\e(s)^\top\h Y_\e(s)+\ti Q_\e(s)\h X_\e(s)\big\}ds+\h Z_\e(s)dW(s),\\
&\h X_\e(t)=\eta,\qq \h Y_\e(T)= \ti G_\e \h X_\e(T),
\end{aligned}\right.\ee
respectively.
Note that $(\bar X_\e,\bar Y_\e,\bar Z_\e)$ (respectively, $(\h X_\e,\h Y_\e,\h Z_\e)$)
depends linearly on $u$ (respectively, $\eta$).
We define two linear operators $\cN_{t,\e}:\cU[t,T]\to\cU[t,T]$
and $\cL_{t,\e}:L^2_{\cF_t}(\Om;\dbR^n)\to\cU[t,T]$ as follows:
\begin{align}
\label{def-cN}
&[\cN_{t,\e} u](s)=\bar Z_\e(s)+ \ti S(s)\bar X_\e(s) +\ti R(s)u(s),
\q s\in[t,T], \q\forall\, u\in\cU[t,T];\\
\label{def-cL}
&[\cL_{t,\e} \eta](s)=\h Z_\e(s)+ \ti S(s)\h X_\e(s),
\qq\qq\q s\in[t,T], \q\forall\, \eta\in L^2_{\cF_t}(\Om;\dbR^n).
\end{align}
Since $\ti A_\e,\ti Q_\e, \ti G_\e$ are uniformly bounded by \ref{ass:H4},
using the standard estimates of FSDEs (\cite[Theorem 6.16, Chapter 1]{Yong-Zhou 1999})
and BSDEs (\autoref{lmm:well-posedness-BSDE}),  it is clear to see that
$$
\|\cN_{t,\e} u\|^2\les K \dbE\int_t^T\[|\bar X_\e(s)|^2+|\bar Z_\e(s)|^2+|u(s)|^2\]ds
\les K \dbE\int_t^T|u(s)|^2 ds,\q  \forall u\in\cU[t,T],
$$
where $K$ is a constant independent of $\e$ and $u$.
Thus the linear operator $\cN_{t,\e}$ is  uniformly bounded with respect to $\e$.
Similarly, we can get the uniform boundedness of $\cL_{t,\e}$;
that is
\bel{uni-bounded-L}
\|\cL_{t,\e} \eta\|^2\les K \dbE|\eta|^2,\q  \forall \eta\in L_{\cF_t}^2(\Om;\dbR^n).
\ee
Let $(M_\e,H_\e)$ be the adapted solution to the following BSDE
\bel{BSDE-L-M}\left\{
\begin{aligned}
dM_\e(s)& = -\big\{M_\e(s)\ti A_\e(s)+\ti A_\e(s)^\top M_\e(s)+\ti Q_\e(s)\big\}ds
            +H_\e(s)dW(s),\\
M_\e(T)&=\ti G_\e.
\end{aligned}\right.
\ee
By \cite[Theorem 3.4]{Sun-Xiong-Yong 2018},
the cost functional $\ti J_\e(t,\eta;u)$ admits the following representation:
\bel{ti-J-rep}
\ti J_\e(t,\eta;u)=\lan \cN_{t,\e} u,u\ran+2\lan\cL_{t,\e} \eta,u\ran+\dbE\lan M_\e(t)\eta,\eta\ran.
\ee
By \ref{ass:H4}--\rf{Prop-VS-1}, \autoref{lmm:well-posedness-BSDE}, and dominated convergence theorem, we have
\begin{align}
\nn&\lim_{\e\to 0^+}\dbE\Big[\sup_{s\in[t,T]}|M_\e(s)-M_0(s)|^2\Big]\\
\label{M-e-M}&\q\les K\lim_{\e\to 0^+}\dbE\int_t^T\[|M_0(s)|^2|\ti A_\e(s)-\ti A_0(s)|^2
           +|\ti Q_\e(s)-\ti Q_0(s)|^2\]ds=0.
\end{align}
Note that by  \ref{ass:H4} and \cite[Proposition 2.2]{Sun-Xiong-Yong 2018},
$M_\e$ is uniformly bounded with respect to $\e$.
Thus by  dominated convergence theorem again, the above implies that
\bel{limit-EM}
\lim_{\e\to0^+}\dbE\lan M_\e(t)\eta,\eta\ran=\dbE\lan M_0(t)\eta,\eta\ran.
\ee
For any given $u\in\cU[t,T]$, by \rf{Prop-VS-1} and the standard estimates of SDEs and BSDEs, we have
\bel{limit-X}
\lim_{\e\to 0^+}\dbE\Big[\sup_{s\in[t,T]}|\bar X_\e(s)-\bar X_0(s)|^2\Big]
\les K \lim_{\e\to 0^+}\dbE \int_t^T\Big |\ti A_\e(s)\bar X_0(s)-\ti A_0(s)\bar X_0(s)\Big|^2ds=0,
\ee
and
\begin{align}
\nn&\lim_{\e\to 0^+}\dbE\Big[\sup_{s\in[t,T]}|\bar Y_\e(s)-\bar Y_0(s)|^2\Big]
+\lim_{\e\to 0^+}\dbE\int_t^T|\bar Z_\e(s)-\bar Z_0(s)|^2ds\\
\nn&\q\les K \lim_{\e\to 0^+}\dbE \int_t^T \Big|\ti A_\e(s)^\top\bar Y_0(s)+\ti Q_\e(s)\bar X_\e(s)
-\ti A_0(s)^\top\bar Y_0(s)-\ti Q_0(s)\bar X_0(s)\Big|^2ds\\
\nn&\hp{\q\les}+K \lim_{\e\to 0^+}\dbE\big|\ti G_\e \ti X_\e(T)-\ti G_0 \ti X_0(T)\big|^2\\
\label{limit-Y}&\q=0.
\end{align}
Combining \rf{limit-X} with \rf{limit-Y}, by the definition  \rf{def-cN} of $\cN_{t,\e}$, we get
\bel{limit-cN}
\lim_{\e\to0^+}\|\cN_{t,\e} u-\cN_{t,0}u\|=0,\q\forall u\in\cU[t,T].
\ee
Similarly, we have
\bel{limit-cL}
\lim_{\e\to0^+}\|\cL_{t,\e} \eta-\cL_{t,0}\eta\|=0,\q\forall \eta\in L^2_{\cF_t}(\Om;\dbR^n).
\ee
For any $u\in\cU[t,T]$, by the representation \rf{ti-J-rep} of the cost functional and \ref{ass:H4},
we have
\begin{align}
\nn\lan \cN_{t,\e} u,u\ran&=\ti J_\e(t,0;u)\\
\nn&=\dbE \int_t^T\[\big\lan\ti Q_\e(s)\bar X_\e(s),\bar X_\e(s)\big\ran+2\big\lan \ti S(s)\bar X_\e(s),u(s)\big\ran
       +\big\lan \ti R(s)u(s),u(s)\big\ran\]ds\\
\nn  &\hp{=}+\dbE\big\lan \ti G_\e \bar X_\e(T),\bar X_\e(T)\big\ran\\
\nn&\ges\dbE \int_t^T\[\big\lan\ti Q_\e(s)\bar X_\e(s),\bar X_\e(s)\big\ran+2\big\lan \ti S(s)\bar X_\e(s),u(s)\big\ran
       +\big\lan \ti R(s)u(s),u(s)\big\ran\]ds\\
\nn&= \dbE \int_t^T\[\big\lan(\ti Q_\e-\ti S^\top \ti R^{-1}\ti S)\bar X_\e,\bar X_\e\big\ran
      +\big\lan \ti R(u + \ti R^{-1}\ti S\bar X_\e),u + \ti R^{-1}\ti S\bar X_\e\big\ran \]ds\\
\label{Je-uni-convex}&\ges\d \dbE \int_t^T\big|u + \ti R^{-1}\ti S\bar X_\e\big|^2ds.
\end{align}
Define a linear operator $\cT_\e:\cU[t,T]\to\cU[t,T]$  by
$$
\cT_\e u=u + \ti R^{-1}\ti S\bar X_\e,\q u\in\cU[t,T].
$$
Then $\cT_\e$ is uniformly bounded and bijective,
with its inverse $\cT_\e^{-1}$ given by
$$
\cT_\e^{-1} u=u - \ti R^{-1}\ti S\ti X_\e,\q u\in\cU[t,T],
$$
where $\ti X_\e$ is the solution of
\bel{ti-X-e}\left\{\begin{aligned}
d\ti X_\e(s) &=\ti A_\e(s)\ti X_\e(s)ds+ \big\{- \ti R^{-1}(s)\ti S(s)\ti X_\e(s)+u(s)\big\} dW(s),\\
 \bar X_\e(t)&=0.
\end{aligned}\right.\ee
Since $\ti A_\e$ is uniformly bounded,
$\|\cT_\e^{-1}\|$ is clearly uniformly bounded with respect to $\e$. 
Thus, we have
\bel{u-S-u}
\dbE \int_t^T\big|u + \ti R^{-1}\ti S\bar X_\e\big|^2ds
=\|\cT_\e u\|^2\ges {1\over \|\cT_\e^{-1}\|^2}\|\cT_\e^{-1}T_\e u\|^2
={1\over \|\cT_\e^{-1}\|^2}\| u\|^2>\gamma\| u\|^2,
\ee
with $\g={1\over \sup_{\e>0}\|\cT_\e^{-1}\|^2}>0$.
Substituting the above into \rf{Je-uni-convex}, we get
\bel{Je-uni-convex-final}
\lan \cN_{t,\e} u,u\ran=\ti J_\e(t,0;u)\ges\d\g \dbE \int_t^T|u(s) |^2ds.
\ee
It follows that
\bel{N-e-inverse}
\|\cN_{t,\e}^{-1}\|\les {1\over \d\g},
\ee
where $\cN_{t,\e}^{-1}$ is the inverse of $\cN_{t,\e}$.
Then by \rf{limit-cN}, we get
\begin{align}
\nn\lim_{\e\to0^+}\left\|\cN_{t,\e}^{-1}u-\cN_{t,0}^{-1}u\right\|
\nn&=\lim_{\e\to0^+}\left\|\cN_{t,\e}^{-1}\cN_{t,0}\cN_{t,0}^{-1}u-\cN_{t,\e}^{-1}\cN_{t,\e}\cN_{t,0}^{-1}u\right\|\\
\nn &\les\lim_{\e\to0^+}\left\|\cN_{t,\e}^{-1}\big\|\big\|\cN_{t,0}\cN_{t,0}^{-1}u-\cN_{t,\e}\cN_{t,0}^{-1}u\right\|\\
\label{N-e-inverse-con}&\les{1\over \d\g}\lim_{\e\to0^+}\left\|\cN_{t,0}\cN_{t,0}^{-1}u-\cN_{t,\e}\cN_{t,0}^{-1}u\right\|
=0,\q\forall u\in\cU[t,T].
\end{align}
By \cite[Corollary 3.5]{Sun-Xiong-Yong 2018}, the (unique) optimal control of Problem (SLQ$_\e$) for $\eta$ is given by
$$
u^*_\e=-\cN_{t,\e}^{-1}\cL_{t,\e} \eta.
$$
Substituting the above into \rf{ti-J-rep} yields that
\bel{ti-V-rep}
\ti V_\e(t,\eta)=\ti J_\e(t,\eta;u^*_\e)
=-\big\lan\cN_{t,\e}^{-1}\cL_{t,\e} \eta,\cL_{t,\e} \eta\big\ran
+\dbE\big\lan M_\e(t)\eta,\eta\big\ran, \q  \eta\in L^2_{\cF_t}(\Om;\dbR^n).
\ee
Then combining the above with \rf{uni-bounded-L}--\rf{N-e-inverse}, we get
\begin{align*}
\big|\ti V_\e(t,\eta)-\ti V_0(t,\eta)\big|
&=\big|-\big\lan\cN_{t,\e}^{-1}\cL_{t,\e} \eta,\cL_{t,\e} \eta\big\ran
+\big\lan\cN_{t,0}^{-1}\cL_{t,0} \eta,\cL_{t,0} \eta\big\ran+\dbE\big\lan M_\e(t)\eta,\eta\big\ran\\
&\hp{=\q}
-\dbE\big\lan M_0(t)\eta,\eta\big\ran\big|\\
&=\big|-\big\lan\cN_{t,\e}^{-1}\cL_{t,\e} \eta,\cL_{t,\e} \eta-\cL_{t,0} \eta\big\ran
-\big\lan\cN_{t,\e}^{-1}(\cL_{t,\e} \eta-\cL_{t,0} \eta),\cL_{t,0} \eta\big\ran\\
&\hp{=\q}
-\big\lan(\cN_{t,\e}^{-1}-\cN_{t,0}^{-1})\cL_{t,0} \eta,\cL_{t,0} \eta\big\ran
+\dbE\big\lan M_\e(t)\eta,\eta\big\ran-\dbE\big\lan M_0(t)\eta,\eta\big\ran\big|\\
&\les\Big\{\big\|\cN_{t,\e}^{-1}\big\|\times\big\|\cL_{t,\e}\big\|\times\big \|\eta\big\|
\times\big\|\cL_{t,\e} \eta-\cL_{t,0} \eta\big\|+\big\|\cN_{t,\e}^{-1}\big\|
\times \big\|\cL_{t,\e} \eta-\cL_{t,0} \eta\big\|\\
&\hp{=\q}
\times\big\|\cL_{t,0}\big\|\times \big\|\eta\big\|
+ \big\|\cN^{-1}_{t,\e}\cL_{t,0} \eta-\cN^{-1}_{t,0}\cL_{t,0} \eta\big\|
\times\big\|\cL_{t,0}\big\| \times\big\|\eta\big\|\Big\}\\
&\hp{=\q}
+\big|\dbE\big\lan M_\e(t)\eta,\eta\big\ran-\dbE\big\lan M_0(t)\eta,\eta\big\ran\big|\\
&\les K \big\|\eta\big\|\Big\{ \big\|\cL_{t,\e} \eta-\cL_{t,0} \eta\big\|
+\big \|\cN^{-1}_{t,\e}\cL_{t,0} \eta-\cN^{-1}_{t,0}\cL_{t,0} \eta\big\|\Big\}\\
&\hp{=\q}
+\big|\dbE\big\lan M_\e(t)\eta,\eta\big\ran-\dbE\big\lan M_0(t)\eta,\eta\big\ran\big|.
\end{align*}
Therefore, by  \rf{limit-EM}--\rf{limit-cL}--\rf{N-e-inverse-con}, we have
\begin{align*}
\lim_{\e\to0^+}\big|\ti V_\e(t,\eta)-\ti V_0(t,\eta)\big|
&\les K \lim_{\e\to0^+}\big\|\eta\big\|\Big\{ \big\|\cL_{t,\e} \eta-\cL_{t,0} \eta\big\|
+\big \|\cN^{-1}_{t,\e}\cL_{t,0} \eta-\cN^{-1}_{t,0}\cL_{t,0} \eta\big\|\Big\}\\
&\hp{=}
+\lim_{\e\to0^+}\big|\dbE\big\lan M_\e(t)\eta,\eta\big\ran-\dbE\big\lan M_0(t)\eta,\eta\big\ran\big|\\
&=0,\qq \forall\, \eta\in L^2_{\cF_t}(\Om;\dbR^n).
\end{align*}
\end{proof}

We now are ready to state and prove the main result of this section.

\begin{theorem}\label{solvability-Riccati-equation}
Let {\rm\ref{ass:H1}--\ref{ass:H2}--\ref{ass:H3}} hold.
Then  Riccati equation \rf{BLQ-Riccati-Equation} has a solution
$(\Si,\L)\in L^\i_{\dbF}(\Om;C([0,T];\dbS^n_+))\times L^2_{\dbF}(0,T;\dbS^{ n})$.
\end{theorem}

\begin{proof}
For any $\e>0$, we consider the following perturbed equation of \rf{BLQ-Riccati-Equation}:
\bel{BLQ-Riccati-Equation-e}
\left\{
\begin{aligned}
d\Si_\e(s)&=\[\Si_\e A^\top+A\Si_\e+\Si_\e Q\Si_\e-BR^{-1}B^\top+\L_\e N(I_n+\Si_\e N)^{-1}\L_\e\\
&\hp{=\[}
-C(I_n+\Si_\e N)^{-1}\Si_\e C^\top-C(I_n+\Si_\e N)^{-1}\L_\e-\L_\e(I_n+N\Si_\e)^{-1}C^\top\]ds\\
&\hp{=}
-\L_\e dW(s),\q s\in[0,T],\\
\Si_\e(T)&=\e I_n.
\end{aligned}\right.
\ee
For the given $\e>0$,
let $(P_\e,\L_\e)\in L^\i_{\dbF}(\Om;C([0,T];\dbS_+^n))\times L^2_{\dbF}(0,T;\dbS^n)$
be the unique solution of Riccati equation \rf{Riccati-Pe}.
By \autoref{P-e-UP}, $P_\e$ is uniformly positive definite (for the given $\e$).
Hence $P_\e$ is invertible, and its inverse $P_\e^{-1}$ is positive definite and bounded (for the given $\e$).
Let
\bel{def-Si-e}
(\Si_\e,\L_\e)\equiv(P^{-1}_\e,P^{-1}_\e\Pi_\e P^{-1}_\e),
\ee
then $(\Si_\e,\L_\e)\in L^\i_{\dbF}(\Om;C([0,T];\dbS^n_+))\times L^2_{\dbF}(0,T;\dbS^n)$.
We shall show that $(\Si_\e,\L_\e)$ defined by \rf{def-Si-e} is a solution of \rf{BLQ-Riccati-Equation-e}.
Using the fact that
$$
0=d(\Si_\e P_\e)=d\Si_\e P_\e+\Si_\e dP_\e+d\Si_\e dP_\e,
$$
we have
\bel{d-Sie}
d\Si_\e= -\Si_\e dP_\e P_\e^{-1}-d\Si_\e dP_\e P_\e^{-1}.
\ee
For convenience, we denote
\bel{Si-e-M}
d\Si_\e(s)=(I)ds-(II)dW(s).
\ee
By \rf{d-Sie}--\rf{Riccati-Pe}, the diffusion term  in \rf{Si-e-M} is given by
\bel{M-e}
(II)=\Si_\e\Pi_\e P_\e^{-1}= P_\e^{-1}\Pi_\e P_\e^{-1}=\L_\e,
\ee
and then the drift term in \rf{Si-e-M} reads
\begin{align}
\nn (I)&=\Si_\e\big\{P_\e A+A^T P_\e+Q-(P_\e C+\Pi_\e)(N+P_\e)^{-1}(C^\top P_e+\Pi_\e)\\
\nn&\hp{=\Si_\e\big\{}
-P_\e BR^{-1}B^\top P_\e\big\}P_\e^{-1}+\Si_\e\Pi_\e P_\e^{-1}\Pi_\e P_\e^{-1}\\
\nn&=A\Si_\e+\Si_\e A^\top+\Si_\e Q\Si_\e-C(N+P_\e)^{-1}C^\top-C(N+P_\e)^{-1}\Pi_\e \Si_\e\\
\nn&\hp{=} -\Si_\e\Pi_\e (N+P_\e)^{-1}C^\top-\Si_\e\Pi_\e(N+P_\e)^{-1}\Pi_\e\Si_\e-BR^{-1}B^\top+\Si_\e\Pi_\e P_\e^{-1}\Pi_\e\Si_\e\\
\nn&=A\Si_\e+\Si_\e A^\top+\Si_\e Q\Si_\e-C(N+P_\e)^{-1}C^\top-C(N+P_\e)^{-1}\Pi_\e \Si_\e\\
\nn&\q -\Si_\e\Pi_\e (N+P_\e)^{-1}C^\top+\Si_\e\Pi_\e [P_\e^{-1}(N+P_\e)-I_n](N+P_\e)^{-1}\Pi_\e\Si_\e-BR^{-1}B^\top\\
%
%
%
\nn&=A\Si_\e+\Si_\e A^\top+\Si_\e Q\Si_\e-C(N+P_\e)^{-1}C^\top-C(N+P_\e)^{-1}\Pi_\e \Si_\e\\
\label{K-e}&\hp{=} -\Si_\e\Pi_\e (N+P_\e)^{-1}C^\top+\Si_\e\Pi_\e\Si_\e N(N+P_\e)^{-1}\Pi_\e\Si_\e-BR^{-1}B^\top.
\end{align}
By the definitions of $\Si_\e,\L_\e$ and using the facts that
\begin{align*}
(N+P_\e)^{-1}=(I_n+P_\e^{-1} N)^{-1}P^{-1}_\e=(I_n+\Si_\e N)^{-1}\Si_\e;\\
(N+P_\e)^{-1}=P^{-1}_\e(I_n+ N P_\e^{-1})^{-1}=\Si_\e(I_n+N\Si_\e)^{-1},
\end{align*}
we can rewrite \rf{K-e} as follows:
\begin{align}
\nn(I)&=A\Si_\e+\Si_\e A^\top+\Si_\e Q\Si_\e-C(I_n+\Si_\e N)^{-1}\Si_\e C^\top-C(I_n+\Si_\e N)^{-1}\Si_\e\Pi_\e \Si_\e\\
\nn&\hp{=} -\Si_\e\Pi_\e \Si_\e(I_n+N\Si_\e)^{-1}C^\top+\Si_\e\Pi_\e\Si_\e N(I_n+\Si_\e N)^{-1}\Si_\e\Pi_\e\Si_\e-BR^{-1}B^\top\\
\nn&=A\Si_\e+\Si_\e A^\top+\Si_\e Q\Si_\e-C(I_n+\Si_\e N)^{-1}\Si_\e C^\top-C(I_n+\Si_\e N)^{-1}\L_\e\\
\label{K-e1}&\hp{=} -\L_\e(I_n+N\Si_\e)^{-1}C^\top+\L_\e N(I_n+\Si_\e N)^{-1}\L_\e-BR^{-1}B^\top.
\end{align}
Note that $\Si_\e$ satisfies the terminal condition
$\Si_\e(T)=P^{-1}_\e(T)=\e I_n.$
Substituting \rf{K-e1}--\rf{M-e} into \rf{Si-e-M}, then it is clearly seen that
$(\Si_\e,\L_\e)$ defined by \rf{def-Si-e} satisfies  equation \rf{BLQ-Riccati-Equation-e}.

\ms

By \cite[Theorem 5.2 ]{Sun-Xiong-Yong 2018}, we have
\bel{P-1-2}
P_{\e_1} \les P_{\e_2},\q \forall\, 0< \e_2\les \e_1<\i.
\ee
Note that for a given $\e_0>0$,
\autoref{P-e-UP} shows that  there exists a constant   $\a_0>0$ such that
\bel{P-e0-e1}
\a_0 I_{n}\les P_{\e_0}.
\ee
Combining the above with \rf{P-1-2}, we get
\bel{P-e0-e12}
\a_0 I_{n}\les P_{\e_0}\les P_\e,\q\forall\, 0< \e\les \e_0.
\ee
Since $\Si_\e$ is the inverse of $P_\e$, the above implies that
\bel{Si-e0-e}
0\les\Si_\e\les \Si_{\e_0}\les {1\over\a_o} I_{n},\q \forall\, 0<\e\les \e_0.
\ee
Then by monotone convergence theorem,
there exists a $\Si\in L^\i_{\dbF}(0,T;\dbS_+^n)$ such that
\bel{Si-e-Si}
\lim_{\e\to 0^+}\Si_\e(s)=\Si(s),\q\as,\,\ae \,s\in[0,T].
\ee
We emphasize that in general the above boundedness  and monotonicity of $\Si_\e$ could not yield
\bel{Si-e-Si1}
\lim_{\e\to 0^+}\esssup_{s\in[0,T]}|\Si_\e(s)-\Si(s)|=0,\q\as
\ee

\ss

Next, we apply \autoref{stability-value-function-SLQ} to get the existence of the diffusion term
$\L$ in \rf{BLQ-Riccati-Equation} by introducing the following method of undetermined coefficients.
Consider the Riccati equation
\bel{new-SLQ-RE}
\left\{
\begin{aligned}
d\ti P(s)&=-\big\{\ti P\ti A+\ti A^\top\ti P+\ti Q-(\ti S^\top+\ti \Pi)(\ti R+\ti P)^{-1}(\ti S+\ti \Pi )\big\}ds\\
&\q+\ti\Pi dW(s),\qq s\in[0,T],\\
\ti P(T)&=\ti G,
\end{aligned}\right.
\ee
where
\begin{align}
\nn&\ti A=-A^\top-Q \Si ,\q \ti R=N^{-1},\q \ti S=N^{-1}C^\top,\q \ti G=0,\\
\label{ti-coefficients}&\ti Q=BR^{-1}B^\top+\Si Q\Si+C(I_n+\Si N)^{-1}\Si C^\top+C(I_n+\Si N)^{-1}N^{-1} C^\top.
\end{align}
By \ref{ass:H1}--\ref{ass:H2}--\ref{ass:H3} and the fact that $\Si\in L^\i_{\dbF}(0,T;\dbS_+^n)$,
$\ti A, \ti R, \ti Q, \ti S$ are bounded.
Moreover, by \ref{ass:H2}--\ref{ass:H3}, we have
\bel{ti-R1}
\ti R=N^{-1}\ges {1\over \l} I_m\gg 0,
\ee
and
\begin{align}
\nn\ti Q-\ti S^\top \ti R^{-1}\ti S&=BR^{-1}B^\top+\Si Q\Si+C(I_n+\Si N)^{-1}\Si C^\top\\
\nn&\hp{=}+C(I_n+\Si N)^{-1}N^{-1} C^\top-CN^{-1}NN^{-1}C^\top\\
\nn&=BR^{-1}B^\top+\Si Q\Si+CN^{-1}(N^{-1}+\Si )^{-1}\Si C^\top\\
\nn&\hp{=}+CN^{-1}(N^{-1}+\Si)^{-1}N^{-1} C^\top-CN^{-1}C^\top\\
\nn&=BR^{-1}B^\top+\Si Q\Si+CN^{-1}(N^{-1}+\Si )^{-1}(\Si +N^{-1})C^\top-CN^{-1}C^\top\\
\label{uniformly-convex}&=BR^{-1}B^\top+\Si Q\Si\ges 0.
\end{align}
Thus by \cite[Proposition 3.5]{Sun-Li-Yong 2016} and \cite[Theorem 6.2]{Sun-Xiong-Yong 2018},
Riccati equation \rf{new-SLQ-RE} admits a unique solution
$(\ti P,\ti\Pi)\in L^\i_{\dbF}(\Om;C([0,T];\dbS^n_+))\times L^2_{\dbF}(0,T;\dbS^n)$.
We claim that if
\bel{ti-P-Si}\ti P=\Si,\ee
then  $(\ti P,-\ti\Pi)$ is a solution of \rf{BLQ-Riccati-Equation}.
In fact, with the equality \rf{ti-P-Si}, equation \rf{new-SLQ-RE} can be rewritten as
\begin{align}
\nn d\ti P(s)&=-\big\{\ti P\ti A+\ti A^\top\ti P+\ti Q-(\ti S^\top+\ti \Pi)(\ti R+\ti P)^{-1}(\ti S+\ti \Pi )\big\}ds+\ti\Pi dW(s)\\
&\nn=-\Big\{\ti P(-A^\top- Q\Si)+(-A-\Si Q)\ti P+BR^{-1}B^\top+\Si Q\Si+C(I_n+\Si N)^{-1}\Si C^\top\\
&\nn\hp{=-\Big\{}
+C(I_n+\Si N)^{-1}N^{-1} C^\top-(CN^{-1}+\ti \Pi)(N^{-1}+\ti P)^{-1}(N^{-1}C^\top+\ti \Pi)\Big\}ds+\ti\Pi dW(s)\\
\nn&=\Big\{A\ti P+\Si Q\ti P+\ti PA^\top+\ti PQ\Si-BR^{-1}B^\top-\Si Q\Si-C(I_n+\Si N)^{-1}\Si C^\top\\
&\nn\hp{=\Big\{}
-CN^{-1}(N^{-1}+\Si)^{-1}N^{-1} C^\top+CN^{-1}(N^{-1}+\ti P)^{-1}N^{-1}C^\top+\ti\Pi(N^{-1}+\ti P)^{-1}\ti\Pi\\
&\nn\hp{=\Big\{}
+\ti\Pi (N^{-1}+\ti P)^{-1}N^{-1}C^\top+CN^{-1} (N^{-1}+\ti P)^{-1}\ti\Pi\Big\}ds+\ti\Pi dW(s)\\
\nn&=\Big\{A\ti P+\ti PA^\top+\ti PQ\Si-BR^{-1}B^\top-C(I_n+\Si N)^{-1}\Si C^\top+\ti\Pi(N^{-1}+\ti P)^{-1}\ti\Pi\\
&\nn\hp{=\Big\{}
+\ti\Pi (N^{-1}+\ti P)^{-1}N^{-1}C^\top+CN^{-1} (N^{-1}+\ti P)^{-1}\ti\Pi\Big\}ds+\ti\Pi dW(s)\\
\nn&=\Big\{A\ti P+\ti PA^\top+\ti PQ\ti P-BR^{-1}B^\top-C(I_n+\ti P N)^{-1}\ti P C^\top+(-\ti\Pi) N(I_n+\ti P N)^{-1}(-\ti\Pi)\\
\label{new-SLQ-RE1}&\hp{=\Big\{}
-(-\ti\Pi)(I_n+N\ti P)^{-1}C^\top-C(I_n+\ti PN)^{-1}(-\ti\Pi)\Big\}ds-(-\ti\Pi) dW(s),
\end{align}
which implies that $(\ti P,-\ti\Pi)$ satisfies  equation \rf{BLQ-Riccati-Equation}.
Thus it suffices to verify that the equality \rf{ti-P-Si} holds.

\ms

Similar to the arguments in \rf{new-SLQ-RE1}, we can obtain that
$(\ti P_\e,\ti\Pi_\e)\equiv(\Si_\e,-\L_\e)$ satisfies the following Riccati equation
\bel{new-SLQ-RE-e}
\left\{
\begin{aligned}
d\ti P_\e(s)
&=-\big\{\ti P_\e\ti A_\e+\ti A_\e^\top\ti P_\e+\ti Q_\e-(\ti S^\top+\ti \Pi_\e)(\ti R+\ti P_\e)^{-1}(\ti S+\ti \Pi_\e )\big\}ds\\
&\q+\ti\Pi_\e dW(s),\qq s\in[0,T],\\
\ti P_\e(T)&=\ti G_\e,
\end{aligned}\right.
\ee
with
\begin{align}
\nn&\ti A_\e=-A^\top-Q\Si_\e ,\q \ti R=N^{-1},\q \ti S=N^{-1}C^\top,\q \ti G_\e=\e I_n,\\
\label{ti-coefficients-e}
&\ti Q_\e=BR^{-1}B^\top+\Si_\e Q\Si_\e+C(I_n+\Si_\e N)^{-1}\Si_\e C^\top+C(I_n+\Si_\e N)^{-1}N^{-1} C^\top.
\end{align}
By \rf{Si-e0-e}, $\ti G_\e,\ti A_\e,\ti Q_\e$ are uniformly bounded (with respect to $0\les\e\les \e_0$).
Moreover, by \rf{Si-e-Si}, we have
\begin{align}
\nn&\lim_{\e\to 0^+}\ti G_\e=\ti G_0\equiv\ti G,\q
\lim_{\e\to 0^+}\ti A_\e(s)=\ti A_0(s)\equiv\ti A(s),\\
\label{ti-A-Ae}&\lim_{\e\to 0^+}\ti Q_\e(s)=\ti Q_0(s)\equiv\ti Q(s),\q \as,\,\ae\,s\in[0,T].
\end{align}
Note that $(\ti P_\e,\ti \Pi_\e )$ is the solution of the corresponding Riccati equation of Problem (SLQ$_\e$),
which is defined by \rf{def-V-e}.
By \cite[Proposition 5.5]{Sun-Xiong-Yong 2018},
the value function $\ti V_\e$ of Problem (SLQ$_\e$) can be given by
$$
\ti V_\e(t,\eta)=\dbE\lan\ti P_\e(t)\eta,\eta\ran,\q \forall\e\ges0,\q\eta\in L^2_{\cF_t}(\Om;\dbR^n).
$$
Then by  \autoref{stability-value-function-SLQ}, the convergence \rf{ti-A-Ae} implies that
\bel{con-Si-1}
\dbE\lan\ti P(t)\eta,\eta\ran=\ti V_0(t,\eta)=\lim_{\e\to0^+}\ti V_\e(t,\eta)
=\lim_{\e\to 0^+}\dbE\lan\ti P_\e(t)\eta,\eta\ran,\q \forall\eta\in L^2_{\cF_t}(\Om;\dbR^n).
\ee
Using the fact that  $\ti P_\e=\Si_\e$, by \rf{Si-e-Si}--\rf{Si-e0-e}  and dominated convergence theorem, we have
\bel{con-Si-2}
\lim_{\e\to 0^+}\dbE\lan\ti P_\e(t)\eta,\eta\ran
=\lim_{\e\to 0^+}\dbE\lan \Si_\e(t)\eta,\eta\ran
=\dbE\lan \Si(t)\eta,\eta\ran,\q \forall\eta\in L^2_{\cF_t}(\Om;\dbR^n).
\ee
Combining \rf{con-Si-1} with \rf{con-Si-2} together, we get
$$
\dbE\lan \ti P(t)\eta,\eta\ran=\dbE\lan \Si(t)\eta,\eta\ran,
\q \forall\eta\in L^2_{\cF_t}(\Om;\dbR^n).
$$
It follows that
$$\ti P(t)=\Si(t),\q\as,\q \ae\,\,t\in[0,T].$$
Thus the equality \rf{ti-P-Si} holds and the proof is completed.
\end{proof}

\begin{remark}\rm
From the proof of \autoref{solvability-Riccati-equation}, we see that if Riccati equation \rf{BLQ-Riccati-Equation} has a solution $(\Si,\L)$,
it can be rewritten as a Riccati equation associated with some forward LQ problem (see \rf{new-SLQ-RE} with $\ti P=\Si$).
Noticing this fact, to prove the solvability of  Riccati equation \rf{BLQ-Riccati-Equation}, we only need to show \rf{ti-P-Si} holds.
Since the coefficients of \rf{new-SLQ-RE} depend on $\Si$, which is a undetermined variable,
we would like to call the above arguments  a {\it method of undetermined coefficients}.
\end{remark}

With a solved solution $(\Si,\L)$ of Riccati equation \rf{BLQ-Riccati-Equation},
the following result shows that the decoupled system of
BSDE \rf{BLQ-BSDE-unbounded} and FSDE \rf{BLQ-fSDE} is uniquely solvable.

\begin{theorem}\label{solvability-BSDE-unbounded}
Let {\rm\ref{ass:H1}--\ref{ass:H2}--\ref{ass:H3}} hold
and $(\Si,\L)\in L^\i_{\dbF}(\Om;C([0,T];\dbS^n_+))\times L^2_{\dbF}(0,T;\dbS^{ n})$
be a solution of the Riccati equation \rf{BLQ-Riccati-Equation}.
Then for any $\xi\in L^2_{\cF_T}(\Om;\dbR^n)$,
the decoupled system of BSDE \rf{BLQ-BSDE-unbounded} and FSDE \rf{BLQ-fSDE} admits a unique solution
$(\f,\b,X)\in L^2_{\dbF}(\Om;C([t,T];\dbR^n))\times
L^1_{\dbF}(\Om; L^2(t,T;\dbR^{n}))\times L^2_{\dbF}(\Om;C([t,T];\dbR^n))$.
\end{theorem}

\begin{proof}
Let $(X^*,Y^*,Z^*)$ be the unique solution of the coupled FBSDE \rf{Y-X-star-nou}
and $(\Si,\L)$ be a solution of the Riccati equation \rf{BLQ-Riccati-Equation}.
Define
\bel{def-f-b}
\f^*\equiv-\Si X^*-Y^*,\q \b^*\equiv(\Si N+I_n)Z^*-\Si C^\top X^*-\L X^*.
\ee
It is clear to see that
$(\f^*,\b^*)\in L^2_{\dbF}(\Om;C([t,T];\dbR^n))\times L^1_{\dbF}(\Om; L^2(t,T;\dbR^{n}))$.
We shall show that $(\f^*,\b^*)$ is a solution of BSDE \rf{BLQ-BSDE-unbounded}.
In fact, by applying It\^{o}'s formula to $s\mapsto \f^*(s)\equiv-\Si(s) X^*(s)-Y^*(s)$, we have
\begin{align}
\nn&d\f^*(s)\equiv d[-\Si(s) X^*(s)-Y^*(s)]\\
\nn&\q =\Big\{-\big[\Si A^\top+A\Si+\Si Q\Si-BR^{-1}B^\top+\L N(\Si N+I)^{-1}\L-C(I+\Si N)^{-1}\Si C^\top\\
\nn&\hp{\q =\Big\{-\big[}
-C(\Si N+I)^{-1}\L-\L(N\Si+I)^{-1}C^\top\big]X^*-\Si[-A^\top X^*+ QY^*]\\
\nn&\hp{\q =\Big\{}
+\L[-C^\top X^*+ NZ^*]-[AY^*+ BR^{-1}B^\top X^*+CZ^*]\Big\}ds\\
\nn&\hp{\q =}
+\big\{\L X^*-\Si[-C^\top X^*+ NZ^*]-Z^*\big\}dW(s)\\
\nn&\q=\Big\{-A(\Si X^*+Y^*)-\Si Q(\Si X^*+Y^*)-\L N(\Si N+I)^{-1}\L X^*+\L(N\Si+I)^{-1}C^\top X^*\\
\nn&\hp{\q =\Big\{}
-\L C^\top X^*+\L NZ^*+C(I+\Si N)^{-1}\Si C^\top X^*+C(\Si N+I)^{-1}\L X^*-CZ^*\Big\}ds\\
\nn&\hp{\q=}
-\big\{(\Si N+I_n)Z^*-\Si C^\top X^*-\L X^*\big\}dW(s)\\
\nn&\q=\Big\{-(A+\Si Q)(\Si X^*+Y^*)+\L N(\Si N+I)^{-1}[(\Si N+I_n)Z^*-\Si C^\top X^*-\L X^*]\\
\nn&\hp{\q =\Big\{}
-C(\Si N+I)^{-1}[(\Si N+I_n)Z^*-\Si C^\top X^*-\L X^*]\Big\}ds\\
\label{f-f-star}&\hp{\q=}
-\big\{(\Si N+I_n)Z^*-\Si C^\top X^*-\L X^*\big\}dW(s),\q s\in[t,T].
\end{align}
With the definition  \rf{def-f-b} of $(\f^*,\b^*)$, the above can be rewritten as
\begin{align}
\nn d\f^*(s)
&=\big\{(A+\Si Q)\f^*(s)+\L N(\Si N+I)^{-1}\b^*(s)-C(\Si N+I)^{-1}\b^*(s)\big\}ds\\
\label{f-f-star-final}&\qq-\b^*(s)dW(s),\q s\in[t,T].
\end{align}
Since $\f^*$ also satisfies the terminal condition:
$$
\f^*(T)=-\Si(T) X^*(T)-Y^*(T)=-Y^*(T)=-\xi,
$$
$(\f^*,\b^*)$ is a solution of \rf{BLQ-BSDE-unbounded}.
Moreover, \rf{def-f-b} implies that
$$
Y^*=-\Si X^*-\f^*,\q Z^*=(I_n+\Si N)^{-1}[\b^*+\Si C^\top X^*+\L X^*].
$$
Substituting the above into the FSDE in \rf{Y-X-star-nou}, we have
\begin{align}
\nn&d X^*(s)=-\big\{(A^\top+Q\Si)X^*(s)+Q\f^*\big\}ds+\big\{-C^\top X^*(s)+N(I_n+\Si N)^{-1}\b^*\\
\label{X-star-f-star}&\qq\qq\q +N(I_n+\Si N)^{-1}(\L+\Si C^\top)X^*(s)\big\} dW(s),\qq s\in[t,T].
\end{align}
Note that
$$
Y^*(t)=-\Si(t)X^*(t)-\f^*(t)=-\Si(t)G_t Y^*(t)-\f^*(t),
$$
we get
$$
Y^*(t)=-(I_n+\Si(t)G_t)^{-1}\f^*(t),
$$
which implies that
$$
X^*(t)=G_tY^*(t)=-G_t(I_n+\Si(t)G_t)^{-1}\f^*(t)=-(I_n+G_t\Si(t))^{-1}G_t\f^*(t).
$$
Combining the above with \rf{X-star-f-star},
$X^*$ satisfies the FSDE \rf{BLQ-fSDE} with $(\f,\b)$  given by $(\f^*,\b^*)$.

\ms

Let $(\f,\b)\in L^2_{\dbF}(\Om;C([t,T];\dbR^n))\times L^1_{\dbF}(\Om; L^2(t,T;\dbR^{n}))$
be any solution of BSDE \rf{BLQ-BSDE-unbounded}
and $X\in L^2_{\dbF}(\Om;C([t,T];\dbR^n))$ be any solution of FSDE \rf{BLQ-fSDE} corresponding to the given $(\f,\b)$.
Let
$$
Y\equiv-\Si X-\f,\q  Z\equiv(I_n+\Si N)^{-1}[\Si C^\top X+\b+\L X].
$$
By \autoref{decoupling}, $(X,Y,Z)$ is the unique solution of FBSDE \rf{Y-X-star}.
Then the uniqueness of  $(\f,\b,X)$ follows from the uniqueness of the adapted solutions to \rf{Y-X-star} immediately.
\end{proof}

\begin{remark}\rm
By \autoref{solvability-Riccati-equation} and \autoref{solvability-BSDE-unbounded},
the solvabilities of Riccati equation \rf{BLQ-Riccati-Equation},
BSDE \rf{BLQ-BSDE-unbounded} and SDE \rf{BLQ-fSDE} are established.
Then a complete and explicit representation \rf{decoupling-u}
for the  optimal control  of Problem (BLQ) is  obtained,
via the solutions to Riccati equation \rf{BLQ-Riccati-Equation},
BSDE \rf{BLQ-BSDE-unbounded} and SDE \rf{BLQ-fSDE}.
\end{remark}

\section*{Acknowledgements} The authors would like to thank the associate editor and the anonymous referees for
their suggestive comments, which lead to this improved version of the paper.

\end{document}